\documentclass[11pt,a4paper]{article}

\usepackage{fancyhdr}
\usepackage{amsmath}
\usepackage{bm}
\usepackage{esint}
\usepackage{amsfonts}
\usepackage{amsthm}
\usepackage{amssymb}
\usepackage{amsbsy}
\usepackage{verbatim}
\usepackage{graphicx}
\usepackage{mathtools}
\usepackage{url}
\usepackage{hyperref}
\usepackage{mathrsfs}
\usepackage[hmargin = 1in,vmargin=.75in]{geometry}
\usepackage{color}
\usepackage{amstext}
\usepackage{stmaryrd}
\usepackage{enumerate}
\usepackage{algpseudocode}
\usepackage{algorithm}
\usepackage{pifont}
\usepackage{prettyref}
\usepackage{hyperref}
\hypersetup{
    colorlinks = true,
}
\usepackage{makecell}

\usepackage{cellspace} 
\font\msbm=msbm10

\numberwithin{equation}{section}

\theoremstyle{plain}
\newtheorem{theorem}{Theorem}[section]
\newtheorem{lemma}[theorem]{Lemma}
\newtheorem{corollary}[theorem]{Corollary}
\newtheorem{conjecture}[theorem]{Conjecture}
\newtheorem{example}[theorem]{Example}

\newtheorem{proposition}[theorem]{Proposition}
\newtheorem{definition}{Definition}[section]

\newtheorem{remark}[theorem]{Remark}
\def\mathbb#1{\hbox{\msbm{#1}}}

\newcommand{\be}{\boldsymbol{e}}

\newcommand{\br}{\boldsymbol{r}}
\newcommand{\bs}{\boldsymbol{s}}

\newcommand{\bu}{\boldsymbol{u}}
\newcommand{\bv}{\boldsymbol{v}}

\newcommand{\bx}{\boldsymbol{x}}
\newcommand{\by}{\boldsymbol{y}}

\newcommand{\bone}{\boldsymbol{1}}

\newcommand{\BLambda}{\boldsymbol{\Lambda}}
\newcommand{\BDelta}{\boldsymbol{\Delta}}
\newcommand{\BPsi}{\boldsymbol{\Psi}}
\newcommand{\BA}{\boldsymbol{A}}
\newcommand{\BB}{\boldsymbol{B}}
\newcommand{\BC}{\boldsymbol{C}}
\newcommand{\BD}{\boldsymbol{D}}
\newcommand{\BE}{\boldsymbol{E}}

\newcommand{\BG}{\boldsymbol{G}}

\newcommand{\BJ}{\boldsymbol{J}}
\newcommand{\BK}{\boldsymbol{K}}
\newcommand{\BL}{\boldsymbol{L}}
\newcommand{\BM}{\boldsymbol{M}}
\newcommand{\BO}{\boldsymbol{O}}
\newcommand{\BP}{\boldsymbol{P}}
\newcommand{\BQ}{\boldsymbol{Q}}
\newcommand{\BR}{\boldsymbol{R}}
\newcommand{\BS}{\boldsymbol{S}}
\newcommand{\BT}{\boldsymbol{T}}
\newcommand{\BU}{\boldsymbol{U}}
\newcommand{\BV}{\boldsymbol{V}}
\newcommand{\BW}{\boldsymbol{W}}
\newcommand{\BX}{\boldsymbol{X}}
\newcommand{\BY}{\boldsymbol{Y}}
\newcommand{\BZ}{\boldsymbol{Z}}

\newcommand{\BPhi}{\boldsymbol{\Phi}}
\newcommand{\BPi}{\boldsymbol{\Pi}}
\newcommand{\BSigma}{\boldsymbol{\Sigma}}

\newcommand{\btheta}{\boldsymbol{\theta}}

\newcommand{\bzero}{\boldsymbol{0}}

\newcommand{\PP}{\mathcal{P}}

\newcommand{\I}{\boldsymbol{I}}
\newcommand{\RR}{\mathbb{R}}
\newcommand{\lag}{\langle}
\newcommand{\rag}{\rangle}

\newcommand{\eps}{\epsilon}

\newcommand*\diff{\mathop{}\!\mathrm{d}}

\DeclareMathOperator{\VEC}{vec}

\DeclareMathOperator{\Tr}{Tr}
\DeclareMathOperator{\Od}{O}

\DeclareMathOperator{\E}{\mathbb{E}}

\DeclareMathOperator{\diag}{diag}
\DeclareMathOperator{\blkdiag}{blkdiag}
\DeclareMathOperator{\ddiag}{ddiag}

\DeclareMathOperator{\rank}{rank}
\DeclareMathOperator{\argmin}{argmin}

\DeclareMathOperator{\St}{St}
\DeclareMathOperator{\BDG}{BDG}

\definecolor{xl}{RGB}{200,50,120}

\begin{document}
\title{\bf
Improved Global Landscape Guarantees for Low-rank Factorization in Synchronization}
\author{Shuyang Ling}

\long\def\\#1//{}

\date{\today}

\author{Shuyang Ling\thanks{New York University Shanghai. Address: 567 Yangsi Road, Pudong New District, Shanghai, China, 200124. S.L. is (partially) financially supported by the National Key R\&D Program of China, Project Number 2021YFA1002800, National Natural Science Foundation of China No.12571105, Shanghai STCSM Rising Star Program No. 24QA2706200, STCSM General Program No. 24ZR1455300, SMEC AI Initiative Program, and NYU Shanghai Boost Fund.}}

\maketitle

\begin{abstract}
The orthogonal group synchronization problem, which aims to recover a set of $d \times d$ orthogonal matrices from their pairwise noisy products, plays a fundamental role in signal processing, computer vision, and network analysis. In recent years, numerous optimization techniques, such as semidefinite relaxation (SDR) and low-rank (Burer-Monteiro) factorization, have been proposed to address this problem and their theoretical guarantees have been extensively studied. While SDR is provably powerful and exact in recovering the least-squares estimator under certain mild conditions, it is not scalable. In contrast, the low-rank factorization is highly efficient but nonconvex, meaning its iterates may get trapped in local minima. To close the gap, we analyze the low-rank approach and focus on understanding when the associated nonconvex optimization landscape is benign, i.e., free of spurious local minima. Recent works suggest that the benignness depends on the condition number of the Hessian at the global minimizer, but it remains unclear whether sharp guarantees can be achieved. In this work, we consider the low-rank approach which corresponds to an optimization problem over the Stiefel manifold ${\rm St}(p,d)^{\otimes n}$. By formulating the landscape analysis into another convex optimization problem, we provide a unified characterization of the optimization landscape for all parameter pairs $(p,d)$ with $p \geq d+2$ for $d\geq 1$ and $p = d+1$ for $1\leq d\leq 3$ which gives a much improved dependence on the condition number of the Hessian. Our results recover the known sharp state-of-the-art bound for $d=1$ which is extremely useful for characterizing the Kuramoto synchronization, and significantly improved the guarantees for the general case $d \geq 2$ with $p \geq d+2$ over the existing results. The theoretical results are versatile and applicable to a wide range of examples.
\end{abstract}

\section{Preliminaries}

The orthogonal group synchronization problem focuses on  recovering $n$ orthogonal matrices $\{\BO_i\}_{i=1}^n$ of size $d\times d$ from their noisy pairwise measurements $\BA_{ij}$:
\[
\BA_{ij} = \BO_i \BO_j^{\top} + \BDelta_{ij}
\]
where $\Od(d) := \{\BO_i\in\RR^{d\times d}: \BO_i\BO_i^{\top} = \BO_i^{\top}\BO_i = \I_d\}$ and $\BDelta_{ij}$ is the additive noise. It include numerous interesting special examples, e.g. community detection and graph clustering ($\mathbb{Z}_2$-synchronization)~\cite{ABBS14,B18,BBV16,GW95}, object matching in computer vision (permutation group)~\cite{PKS13}, phase retrieval and clock synchronization (circle group)~\cite{WDM15,BBS17,S11}, cryo-EM and point cloud registration (SO(3) group)~\cite{CKS15,SS11,L23b} and it also serves as one of the core subproblems in SLAM and robotics (special Euclidean group SE($d$)-synchronization)~\cite{RCBL19}. 

Recently, there has been significant progress in solving the orthogonal group synchronization problem. In this article, we provide a non-exhaustive review of this problem, with a particular focus on approaches based on optimization methods. Moreover, we will review the connection of the optimization landscape to the condition number of the Hessian at the global minima.

\paragraph{Least squares estimator:}
One common approach is to minimize the least squares objective function:
\[
\min_{\BG_i\in\Od(d)}~\sum_{i,j} \|\BA_{ij} - \BG_i\BG_j^{\top} \|_F^2 =\min -\sum_{i,j} \lag \BA_{ij}, \BG_i\BG_j^{\top}\rag
\]
which is equivalent to minimize a generalized quadratic form:
\begin{equation}\label{def:pp}
\min_{\BG\in\Od(d)^{\otimes n}}~-\lag \BA, \BG\BG^{\top}\rag
\end{equation}
where $\BG\in\Od(d)^{\otimes n}\subseteq \RR^{nd\times d}$ means it is a concatenation of $n$ $d\times d$ orthogonal matrices and $\BA_{ij}$ is the $(i,j)$-block of $\BA\in\RR^{nd\times nd}$. Note that when $d=1$, the problem reduces to the graph Max-Cut problem. Consequently, it is generally NP-hard and, in fact, belongs to Karp's original list of 21 NP-hard problems~\cite{K72} in this case.

\paragraph{SDR-semidefinite relaxation:}
To overcome the computational hardness, we consider the convex relaxation of~\eqref{def:pp} by following the Goemans-Williamson relaxation of the graph Max-Cut problem~\cite{GW95}.
By letting $\BX=\BG\BG^{\top}$ which satisfies $\BX\succeq 0$ and $\BX_{ii}= \I_d$, i.e., the diagonal block equals $\I_d$,  the resulting convex relaxation is given by
\begin{equation}\label{def:sdr}
\min~-\lag \BA, \BX\rag\quad \text{s.t.}\quad \BX\succeq 0,\BX_{ii} = \I_d.
\end{equation}
This convex optimization can be solved by using algorithms such as interior point method~\cite{BN01,NN94}.
Unlike the Max-Cut problem that needs to perform random rounding after solving the SDR, it is very interesting that this Goemans-Williamson type relaxation is tight in many  applications, i.e., the global minimizer to~\eqref{def:sdr} is exactly rank-$d$ and recovers the globally optimal least squares estimator, under several statistical models~\cite{ABBS14,B18,L22,L23,L23b}. 

Using the duality theory in convex optimization~\cite{BN01,BV04}, one can show that the rank-$d$ solution $\BZ\BZ^{\top}$ is a global minimizer~\cite{B23,L23} if
\[
\BL: = \widehat{\BLambda} - \BA\succeq 0,~~\BL\BZ = 0.
\]
In fact, one can show that $\BL$ is essentially related to the Hessian and $\BL\BZ$ is the stationarity condition of the KKT condition.
Here $\widehat{\BLambda}$ is a $nd\times nd$ block-diagonal matrix whose $i$-th $d\times d$ diagonal block is
\[
\widehat{\BLambda}_{ii} =  \BDG(\BA\BZ\BZ^{\top}) = \frac{1}{2}\sum_{j=1}^n \left(\BA_{ij}\BZ_j\BZ_i^{\top} +\BZ_i\BZ_j^{\top}\BA_{ij}^{\top}\right).
\]
where
\[
\BDG(\BX) = \blkdiag\left(\frac{\BX_{11}+\BX_{11}^{\top}}{2},\cdots,\frac{\BX_{nn}+\BX_{nn}^{\top}}{2}\right)
\]
takes any $nd\times nd$ matrix, symmetrize it, and only keep the diagonal blocks.
In fact, $ \widehat{\BLambda}_{ii}$ corresponds to the dual variable associated to $\BX_{ii} = \I_d.$ 
In particularly, if $\lambda_{d+1}(\BL) > 0$, i.e., the $(d+1)$-th smallest eigenvalue of $\BL$ is positive, then $\BZ\BZ^{\top}$ is the unique global minimizer to the least squares objective and the SDR as there are $d$ eigenvalues equal to 0 if $\BL\BZ = 0.$

\paragraph{Burer-Monteiro factorization:}
As the convex relaxation is slow and not scalable, one considers the low-rank factorization to the SDR, known as the Burer-Monteiro factorization~\cite{BM03,BM05}:
\begin{equation}\label{def:bm}
\min_{\BS\in\St(p,d)^{\otimes n}}~-\lag \BA, \BS\BS^{\top}\rag
\end{equation}
where $\BS\in\St(p,d)^{\otimes n}$ and $\BS_i\in\St(p,d)$, i.e.,
\[
\St(p,d) := \{\BS_i\in\RR^{d\times p}: \BS_i\BS_i^{\top} = \I_d\},~~p\geq d.
\]
The Burer-Monteiro factorization~\eqref{def:bm} can be viewed as an interpolation between the least squares estimation ($p=d$) and the SDR ($p= nd$). Many optimization methods can be used to solve this problem~\cite{WY13,AMS08,B23}. 
In practice, despite the nonconvexity of the objective function, local method and gradient-based approach work remarkably well in obtaining an estimator of the orthogonal group~\cite{L22,LYS23,L23b}. Several works have done in understanding the properties of local minima~\cite{MMMO17} for the synchronization problem. 
The core question regarding the performance of the Burer-Monteiro factorization is: 
\begin{align*}
&\text{\em Is the landscape free of spurious local minima if the SDR works?} 
\end{align*}
Here we also refer it to {\em the benignness of landscape}, in which the nonconvex function has no spurious non-global local minimizers, i.e., every local minimizer is global. 
It is well known that the SDR is tight under information-theoretically near-optimal conditions on the signal-to-noise ratio (SNR)~\cite{B18,L22}. In other words, we are interested in whether  the optimization landscape is benign under near-optimal conditions.

\paragraph{Analysis of the nonconvex landscape:}
The landscape analysis of the nonconvex function is a significant mathematical question and has connection with many different scientific areas. Nowadays, the nonconvex function is ubiquitous in numerous problems in signal processing and machine learning. 
The landscape analysis has been studied in the well-known phase retrieval problem~\cite{SQW18,M25b}, matrix completion~\cite{GLM16}, dictionary learning~\cite{SQW16}, statistical physics~\cite{AB13,ABC13},  low-rank matrix sensing problems~\cite{Z24}, and empirical risk minimization~\cite{MBM18}. Most of these aforementioned works such as~\cite{SQW18,GLM16,SQW16,MBM18} focus on understanding how the landscape depends on the data sample size, and aim to derive a near-optimal bound on the sample size to guarantee  a benign landscape that is similar to its population counterpart. The other important direction is to understand how the complexity of the landscape depends on the noise in the data such as~\cite{AB13,AMMN18}, and thus completely characterize the complexity of smooth random functions~\cite{ABC13}.

For the synchronization problem, the landscape analysis on its nonconvex energy function~\eqref{def:bm} has profound implications beyond the optimization perspective. 
In particular, if $p =2$ and $d=1$, the function landscape analysis provides insights into the global synchronization of Kuramoto model on general networks~\cite{ABK22,EW24,K75,LXB19,MB24,M25,L25}. For general $d\geq 2$, it has applications in the orthogonal group synchronization~\cite{MMMO17,L22}, generalized Procrustes problem~\cite{L23b}, and non-Abelian Kuramoto model (quantum Kuramoto model). 
Now we provide a brief discussion on the recent research regarding the benign landscape.
As pointed out earlier, the higher $p$ is, the more likely the optimization landscape is to be benign. Intuitively, increasing $p$ leads to more degree of freedom in the search space, and in the extreme case $p = nd,$ the function is equivalent to the convex optimization. 
Therefore, many studies have been focusing on how the overparameterization in the Burer-Monteiro factorization~\eqref{def:bm} affects the optimization landscape. 
In particular, we are interested in finding the smallest $p$ such that the landscape of~\eqref{def:bm} is benign for the given input data $\BA.$
The work~\cite{BVB20} has established a general theorem that guarantees every second-order critical point is a local minimizer if $p$ is roughly of order $\sqrt{2nd}$ which is known as the Pataki's bound~\cite{P98}. This is later confirmed by~\cite{WW20} that this bound is almost tight for general data matrix $\BA$. 

The Pataki's bound is apparently too conservative in many applications because the data matrix $\BA$ often has certain statistical structures. 
Empirically, one observes that even if $p$ is of constant order, running gradient descent with random initialization works remarkably well. This inspires another stream of research focuses on how the signal-to-noise ratio (SNR) affects the benign landscape of the corresponding Burer-Monteiro factorization. Consider a concrete example in which $\BA_{ij} = \BO_i\BO_j^{\top}+\sigma\BW_{ij}$ for the additive Gaussian noise $\BW_{ij}\in\RR^{d\times d}$. Suppose $\sigma=0$, then it is very simple to verify that the resulting landscape of~\eqref{def:bm} is free of spurious local minimizers. The natural question is what the largest allowable $\sigma$ is so that the optimization landscape is benign such as~\cite{BBV16,L23}. However, they are suboptimal in terms of SNR for any constant $p$, i.e., the bound to guarantee the landscape benignness differs significantly from the near-optimal bound via SDR.

The first result that provides a near-optimal optimal bound on the SNR is given in the author's previous result~\cite{L25}, in which we establish a sufficient condition for the benignness of the landscape that only depends on the condition number of the Riemannian Hessian at the global minimizer. More precisely, note that $\BL\succeq 0$ and $\BL\BZ = 0$, and let $\lambda_{\max}(\BL)$ and $\lambda_{d+1}(\BL)$ be the largest and $(d+1)$-th smallest eigenvalues of $\BL$ respectively. Suppose 
\begin{equation}\label{eq:l25}
\frac{\lambda_{\max}(\BL)}{\lambda_{d+1}(\BL)} \leq \frac{p+d-2}{2d},~~d\geq 1,
\end{equation}
then the optimization landscape is benign, i.e., it is free of spurious local minima. 
The interesting part of this result is that it does not assume that statistical assumption on data generative model, and it applies to numerous examples in $\mathbb{Z}_2$-synchronization and $\Od(d)$-synchronization. For the data obeying certain statistical models, then the condition number of $\BL$ encodes the information of the SNR in the data: {\em the higher the noise, the larger the condition number}. However, as the condition number is at least 1, thus the bound~\eqref{eq:l25} is only meaning if any $p\geq d+3$ for any $d\geq 1$. 
Remarkably, Rakoto Endor and Waldspurger~\cite{EW24} improved the bound for $d=1$ from~\eqref{eq:l25} in~\cite{L25} to
\[
\frac{\lambda_{\max}(\BL)}{\lambda_2(\BL)} < p,
\]
which is proven to be optimal in the $d=1$ case. This is particularly interesting because it covers the case for Kuramoto model, i.e., $d = 1$ and $p=2$, which extends the global synchrony of Kuramoto model on sphere and Stiefel manifold~\cite{MTG17,MB24}. Later, 
McRae extended it to $d=1$ for normalized Laplacian~\cite{M25}.
It can be applied to analyze the synchronization of Kuramoto model, especially the Kuramoto synchronization on graphs~\cite{LXB19,ABK22,MABB24}. 
The relationship between the condition number of the objective function and landscape has also been studied in other low-rank matrix recovery problem such as~\cite{Z24}, where the author establishes a sufficient condition of the benign landscape via the condition number.

Finally, we summarize our contributions. We establish a unified bound that links the benignness of the optimization landscape to the condition number of the Hessian at the global minimizer. In particular, the landscape of~\eqref{def:bm} is free of spurious second-order critical points provided that
\[
\frac{\lambda_{\max}(\BL)}{\lambda_{d+1}(\BL)} < \alpha_G(p,d),
\]
where $\alpha_G(p,d)$ is a scalar depending only on $p$ and $d$. Our bound improves upon all existing ones in the sense that $\alpha_G(p,d)$ is larger than previously known thresholds; consequently, the landscape is provably robust to higher noise levels, since stronger noise typically corresponds to a larger condition number of~$\BL$.
This guarantee holds for any pair $(p,d)$ with $p \ge d+2$, and also for the borderline case $p=d+1$ when $1 \le d \le 3$. When $d=1$, our result matches the state of the art in~\cite{EW24,M25}, which in particular implies global synchronization even at $p=2$. For $d \ge 2$, we substantially improve the bound in~\cite{L25}. As in prior work, our sufficient condition is purely deterministic and applies to a broad range of models, including the Kuramoto model and $\mathbb{Z}_2$-synchronization; moreover, it yields improved condition-number requirements for orthogonal group synchronization and the generalized orthogonal Procrustes problem.
On the technical side, our main contribution is to extend the framework of~\cite{EW24} from $d=1$ and $p\ge 2$ (sphere-constrained variables) to the general regime $d\ge 1$ and $p \ge d+2$ (variables on the Stiefel manifold). More precisely, we cast the estimation of the relevant condition number as a convex optimization problem and derive sharper lower bounds via dual feasibility. The key technical component is the construction of a broader family of dual certificates that remain valid for $d\ge 2$, where $\mathrm{O}(d)$ is noncommutative.

Note that many prior landscape analyses, such as in~\cite{SQW18,GLM16,SQW16}, rely on statistical assumptions on the data-generating models. It remains an intriguing challenge to provide a general framework characterizing the optimization landscape that also applies to specific statistical models. Our work aligns with this goal~\cite{L25,EW24,M25b,Z24}, moving toward a unified and general understanding of nonconvex optimization landscapes. We believe the techniques and methodology developed here, together with recent advances, can be extended to yield a broadly applicable framework for landscape analysis in other nonconvex problems.

\section{Main results}\label{s:main}

Before proceeding, we go over some notations that will be used. We denote boldface $\bx$ and $\BX$ as a vector and matrix respectively, and $\bx^{\top}$ and $\BX^{\top}$ are their corresponding transpose. The $\I_n$ and $\BJ_n$ stands the $n\times n$ identity  and constant ``1" matrix of size $n\times n$. For a vector $\bx$, $\diag(\bx)$ is a diagonal matrix whose diagonal entries are given by $\bx$; for a set of $n$ matrices $\{\BX_i\}_{i=1}^n$, we let $\blkdiag(\BX_1,\cdots,\BX_n)$ be a block diagonal matrix with the diagonal blocks equal to $\{\BX_i\}_{i=1}^n.$
For two matrices $\BX$ and $\BY$ of the same size, $\lag \BX,\BY\rag = \sum_{i,j}X_{ij}Y_{ij} = \Tr(\BX\BY^{\top})$ is their inner product, and $\BX\otimes\BY$ is their Kronecker product. For any matrix $\BX$, the Frobenius and operator norms are denoted by $\|\BX\|_F$ and $\|\BX\|$ respectively. We denote $\lambda_{\max}(\BX)$, $\lambda_k(\BX)$,  $\sigma_k(\BX)$, and $\sigma_{\min}(\BX)$ as the largest eigenvalue, $k$-th smallest eigenvalue, $k$-th smallest singular value, and the smallest singular value of $\BX$ respectively. For any real number $z$, we define $z_+ = \max\{z, 0\}.$ 

\vskip0.25cm

In this section, we will first provide the main result in Theorem~\ref{thm:main} and its simplified but slightly weaker version in Theorem~\ref{thm:alphaM}. In Section~\ref{ss:main1}, we give more details on Theorem~\ref{thm:main}, especially on the exact expression of $G(x,\tau)$, the numerical value of $\alpha_G(p,d)$, and the comparison with the existing work. In Section~\ref{ss:main2}, we briefly provide some applications and Section~\ref{ss:main3} gives some counterexamples which demonstrate how far we are away from the possible optimal bound. Finally, we provide a roadmap of the proof by only giving the main technical results without delving into the details in Section~\ref{ss:roadmap}.
\vskip0.25cm

Now we are ready to present our main results. Before that, we clarify the core question of our interest: assume that the SDR~\eqref{def:sdr} is tight, and the Hessian matrix $\BL$ at the global minimizer satisfies
\[
\BL \BZ = 0,~~~\BL\succeq 0,~~~\lambda_{d+1}(\BL)> 0,
\]
i.e., $\BZ\BZ^{\top}$ is the unique global minimizer to the SDR~\eqref{def:sdr}.
How does the optimization landscape depend on the condition number of the Hessian $\BL$? 
Our main result shows that the optimization landscape is free of spurious second-order critical points. This means as long as $\BL$ is well-conditioned, the landscape is benign.

\begin{theorem}\label{thm:main}{\bf Master theorem.}
Suppose $\BL\BZ = 0$ and $\BL\succeq 0$.
Every second-order critical point $\BS$ is a global minimizer satisfying $\BS\BS^{\top} = \BZ\BZ^{\top}$ if
\[
\frac{\lambda_{\max}(\BL)}{\lambda_{d+1}(\BL)} < \alpha_G(p,d)
\]
where
\begin{equation}\label{def:alpha}
\begin{aligned}
\alpha_G(p,d) & : = \max_{0\leq \tau\leq 1} \alpha_G(p,d,\tau), \\
\alpha_G(p,d,\tau) & : = \max_{\alpha}\{ \alpha: r(\tau) \alpha^{-1} \geq G(1-\alpha^{-1},\tau)   \},
\end{aligned}
\end{equation}
and
\begin{equation}\label{def:rqg}
\begin{aligned}
r(\tau) & := 2(d-1)\tau^2 + p - d, \\
q(\tau) & :=  (d-2)\tau^2 - 2(d-1)\tau + d, \\
G(x,\tau) & := \sup_{0\leq u\leq w\leq v\leq 1}\left\{ q(\tau) \cdot\frac{u}{v} \left( 2 - \frac{pu}{p-d}\right) + d\tau^2 \cdot\frac{w}{v}\left(2 -\frac{pdw}{pd-1} \right)  - \frac{1}{x}\cdot \frac{(1-dv)_+u^2}{dv^2}\right\}.
\end{aligned}
\end{equation}
\end{theorem}

This theorem provides a unified characterization of the relation between the condition number and landscape for any $d\geq 1.$ 
We claim that with this theorem, the optimization landscape of~\eqref{def:bm} is benign under nearly identical conditions that guarantee the performance of the SDR~\eqref{def:sdr}. Here $G(x,\tau)$ is an increasing function in $x>0$ for any fixed $0\leq \tau\leq 1$. We will provide the exact form of $G(x,\tau)$ and the resulting $\alpha_G(p,d,\tau)$ in Proposition~\ref{prop:gfun} and Theorem~\ref{thm:main2}. The roadmap of the proof is provided in Section~\ref{ss:roadmap}.

We give a brief discussion on $r(\tau),q(\tau)$, and $\tau$.
The parameter $\tau\in[0,1]$ parametrizes a family of tangent vectors used in the dual certificate construction, as shown in~\eqref{def:LP}. In particular, if $\tau=1$, the result reduces to the special case in~\cite{L25}. 
The scalar functions $q(\tau)$ and $r(\tau)$ appearing in Theorem~\ref{thm:main} capture the aggregate effect of this family: as shown in~\eqref{eq:lptrace}, they arise as expectations of quadratic expressions associated with the random tangent vectors. Therefore, optimizing over $\tau$ amounts to selecting the most favorable member of this family for establishing dual feasibility and leads to improved dependence on the condition number.

As the exact expression of~\eqref{def:alpha} is quite complicated, we first give a simple and useful corollary and provide a weaker version of Theorem~\ref{thm:main}.

\begin{corollary}\label{cor:alphaM}
Suppose $M(x,\tau)$ is an upper bound of $G(x,\tau)$ for all $0\leq \tau\leq 1$ and $x>0$, then
\begin{equation}\label{def:alphaM0}
\alpha_M( p,d ) := \max_{0\leq \tau\leq 1} \{\alpha: r(\tau) \alpha^{-1} \geq M(1-\alpha^{-1}, \tau)\}
\end{equation}
provides a lower bound of $\alpha_G(p,d)$. Moreover, 
\[
\frac{\lambda_{\max}(\BL)}{\lambda_{d+1}(\BL)} <\alpha_M( p,d )  
\]
guarantees a benign optimization landscape.
\end{corollary}
\begin{proof}
The proof of this corollary is simple: suppose $(\tau_M, \alpha_M(p,d))$ is the maximizer to~\eqref{def:alphaM0}, then
\[
r(\tau_M)\alpha_M(p,d)^{-1} \geq M(1-\alpha_M(p,d)^{-1}, \tau_M) \geq G(1-\alpha_M(p,d)^{-1}, \tau_M)
\]
and thus it holds that $\alpha_{M}(p,d) \leq \alpha_G(p,d).$
\end{proof}

Next, we give a simple and explicit lower bound of $\alpha_G(p,d)$ by finding an upper bound of $G(x,\tau).$ By the definition of $G(x,\tau)$, we can see that 
\begin{equation}
\begin{aligned}
G(x,\tau) & \leq 2( q(\tau) + d\tau^2).
\end{aligned}
\end{equation}
which follows from dropping all the negative terms  in~\eqref{def:rqg} and using $u\leq w\leq v\leq 1.$ This upper bound is clean because it does not involve $x.$
Combining this upper bound of $G(x,\tau)$ with the master theorem and Corollary~\ref{cor:alphaM}, we have the following result.

\begin{theorem}{\bf (Estimation of $\alpha_G(p,d)$ for $d\geq 2,p\geq d+3$)}\label{thm:alphaM}
The function $G(x,\tau)$ has an upper bound
\[
G(x,\tau) \leq 2(q(\tau)+d\tau^2) =:M(x,\tau).
\]
Then the landscape is benign if
\[
\frac{\lambda_{\max}(\BL)}{\lambda_{d+1}(\BL)} < \alpha_M(p,d)
\]
where
\begin{align*}
\alpha_M(p,d) & :=   \max_{0\leq \tau\leq 1}\alpha_M(p,d,\tau),  \\
\alpha_M(p,d,\tau) & : = \frac{r(\tau)}{2(q(\tau)+d\tau^2)}  = \frac{1}{2}\cdot \frac{2(d-1)\tau^2+p-d}{2(d-1)\tau^2-2(d-1)\tau + d}. 
\end{align*}
The value $\alpha_{M}(p,d)$ has an explicit form:
\begin{equation}\label{def:alphaM}
\alpha_M(p,d) : = \frac{\tau_*}{2\tau_{*} - 1},~~~~~
\tau_*(p,d) =  \frac{1}{2}\left( \frac{2d-p}{d-1} + \sqrt{ \frac{(2d-p)^2}{(d-1)^2} + \frac{2(p-d)}{d-1} }\right).
\end{equation}
In particular, it holds that $\tau_*(d+2,d) = 1$ and $\alpha_{M}(d+2,d)=0,~d\geq 2$ and thus the bound provides a non-trivial estimation for $p\geq d+3$, i.e., 
\[
\alpha_M(p,d) > 1,~~\forall p \geq d+3,~d\geq 2.
\]
Moreover, $\alpha_M(p,d)$ is an increasing function in $p$ by its definition.
\end{theorem}
\begin{remark}
Note that the higher the noise level, the larger the condition number of $\BL$, and thus a larger $\alpha_G(p,d)$ means the optimization landscape is more robust against the noise in the data for any given pair of $(p,d)$.
Apparently, this has already improved the existing work~\cite{L25}  by the author because for $\tau=1$, then
\[
\alpha_G(p,d) \geq \alpha_M(p,d) = \max_{0\leq \tau\leq 1} \frac{r(\tau)}{2(q(\tau)+d\tau^2)} > \alpha_M(p,d,1) = \frac{r(1)}{2(q(1)+d)} = \frac{ p + d - 2}{ 2d }
\]
where $q(1) = 0$, $r(1) = p + d-2$ and the right hand side is exactly the bound in~\cite{L25}. 
\end{remark}
\subsection{Exact computation of $\alpha_G(p,d)$} \label{ss:main1}

Corollary~\ref{cor:alphaM} implies a lower bound $\alpha_M(p,d)$ of $\alpha_G(p,d)$. In this section, we will present the exact formula of $G(x,\tau)$ and $\alpha_G(p,d)$.
The proposition gives $G(x,\tau)$ which is a piece-wise function in $x$ for any given $\tau.$
\begin{proposition}[\bf Exact formula of $G(x,\tau)$ in~\eqref{def:rqg}]\label{prop:gfun}
For $d = 1$, then
\begin{equation}\label{def:Gfun1}
G(x,\tau) = \{x(1-\tau^2)^2+2\tau^2\}\vee\frac{p-1}{p}.
\end{equation}
For $d\geq 2$, $p\geq d+1$, then
\begin{equation}\label{def:Gfun}
G(x,\tau) = 
\begin{dcases}
2q(\tau) + 2d\tau^2 - \dfrac{1}{xd}, & {\rm if}~~x dq(\tau) \geq \max\left\{1, \frac{1}{h(\tau)}\right\}, \\
 xdq^2(\tau) + 2d\tau^2, & {\rm if}~~ 2 - h(\tau) \leq xdq(\tau) \leq 1, \\
 2q(\tau)+2d\tau^2 - q(\tau)h(\tau), & {\rm otherwise}.
\end{dcases}
\end{equation}
where
\begin{equation}\label{def:hfun}
h(\tau) := 
\begin{dcases}
2 - \dfrac{d}{d+1} + \dfrac{pd}{pd-1}\cdot \dfrac{\tau^2}{q(\tau)}, & p = d+1, \\
 \dfrac{p}{d(p-d)} +\dfrac{pd}{pd-1}\cdot\frac{\tau^2}{q(\tau)}, & p\geq d+2. \\
 \end{dcases}
\end{equation}
\end{proposition}

By using the expression of $G$ in Proposition~\ref{prop:gfun} with Theorem~\ref{thm:main}, we have the following detailed version of Theorem~\ref{thm:main} by solving the inequality in~\eqref{def:alpha}, i.e.,
\[
r(\tau)\alpha^{-1} \geq G(1-\alpha^{-1},\tau)
\]
for each $0\leq \tau\leq 1.$ Since $G(x,\tau)$ is piecewise in $x$ for any $0\leq \tau\leq 1$, we find the corresponding $\alpha_G(p,d,\tau)$ by searching for the largest $\alpha$ that satisfies $r(\tau)\alpha^{-1} \geq G(1-\alpha^{-1},\tau)$.
The following theorem is a detailed version of Theorem~\ref{thm:main}.

\begin{theorem}[\bf Exact form of $\alpha_G(p,d)$ for $p \geq d+1,d\geq 1$]\label{thm:main2}
For the $G(x,\tau)$ defined in~\eqref{def:Gfun1} and~\eqref{def:Gfun} with $p\geq d+1$, the solution to 
\[
r(\tau)\alpha^{-1} \geq G(1-\alpha^{-1},\tau)
\] 
for fixed $0\leq \tau\leq 1$ is given by the followings.

\noindent For $d=1$, then $\alpha_G(p,1,\tau)$ is given as follows.

\begin{enumerate}
\item $\alpha_G(p,1,\tau) = \alpha$ if $\alpha =  \frac{p-1+(1-\tau^2)^2}{2\tau^2+(1-\tau^2)^2}$ satisfies
\begin{equation}\label{eq:alphaG1:1}
(1-\alpha^{-1})(1-\tau^2)^2+2\tau^2 \geq \frac{p-1}{p};
\end{equation}

\item otherwise, $\alpha_G(p,1,\tau) = p$ if $\alpha =  p$ satisfies
\begin{equation}\label{eq:alphaG1:2}
(1-\alpha^{-1})(1-\tau^2)^2+2\tau^2 \geq \frac{p-1}{p}.
\end{equation}
\end{enumerate}

For $p\geq d+1$ and $d\geq 2$, there are two scenarios.
For $h(\tau) \geq 1$ where $h(\tau)$ is defined in~\eqref{def:hfun}, then
\begin{enumerate}
\item $\alpha_G(p,d,\tau) = \alpha$ if  $\alpha$ is the larger root of the equation $\frac{r(\tau)}{ \alpha}= 2(q(\tau) + d\tau^2) - \frac{1}{(1-\alpha^{-1})d}$, and satisfies
\begin{equation}\label{eq:alphaG2:1}
1 - \alpha^{-1} \geq \frac{1}{dq(\tau)};
\end{equation}

\item otherwise, $\alpha_G(p,d,\tau) = \alpha$ if $\alpha =  \frac{ dq^2(\tau)+r(\tau)    }{ dq^2(\tau) + 2d\tau^2}$ satisfies
\begin{equation}\label{eq:alphaG2:2}
 \frac{2 - h(\tau)}{dq(\tau)} \leq 1-\alpha^{-1}  \leq  \frac{1}{dq(\tau)};
\end{equation}

\item otherwise, $\alpha_G(p,d,\tau) = \alpha$ if $\alpha = \frac{r(\tau)}{2 (q(\tau)+d\tau^2) - q(\tau) h(\tau)  }$ satisfies
\begin{equation}\label{eq:alphaG2:3}
1 - \alpha^{-1} \leq \frac{2 - h(\tau)}{dq(\tau)}.
\end{equation}
\end{enumerate}
For $h(\tau) \leq 1$,  then
\begin{enumerate}
\item $\alpha_G(p,d,\tau) = \alpha$ if $\alpha$  is the larger root of the equation $ \frac{r(\tau)}{ \alpha}= 2(q(\tau) + d\tau^2) - \frac{1}{(1-\alpha^{-1})d}$, and satisfies
\begin{equation}\label{eq:alphaG3:1}
 1 - \alpha^{-1} \geq \frac{1}{dq(\tau)h(\tau)};
\end{equation}

\item otherwise, $\alpha_G(p,d,\tau) = \alpha$ if $\alpha =  \frac{r(\tau)}{2 (q(\tau)+d\tau^2) - q(\tau)h(\tau) }$ satisfies
\begin{equation}\label{eq:alphaG3:2}
  1-\alpha^{-1}  \leq  \frac{1}{dq(\tau)h(\tau)}.
\end{equation}

\end{enumerate}
Moreover, for $\alpha_G(p,d) = \max_{0\leq \tau\leq 1}\alpha_G(p,d,\tau)$, it satisfies
\begin{align*}
&\alpha_G(p,1) = p,~~ \forall p\geq 2, \\
 &\alpha_G(p,d) > 1,~~ \forall p\geq d+2, d\geq 2 ~{\rm or } ~p = d+1,~d=2,3.
\end{align*}
Suppose 
\[
\frac{\lambda_{\max}(\BL)}{\lambda_{d+1}(\BL)} < \alpha_G(p,d), 
\]
then the landscape is benign.
\end{theorem}
\begin{proof}
To demonstrate how to use Proposition~\ref{prop:gfun} and~\eqref{def:alpha} to derive $\alpha_G(p,d)$, we give the proof for  $d=1$. We leave the case of $d\geq 2$ to Section~\ref{ss:proof_5}.
Note that for $d=1$, we have
\[
G(x,\tau) = \{x(1-\tau^2)^2+2\tau^2\}\vee\frac{p-1}{p}.
\]
For each $\tau$, we aim to find out the largest $\alpha_G(p,1,\tau)$ satisfying the inequality below:
\begin{align*}
& \{(1-\alpha^{-1})(1-\tau^2)^2+2\tau^2\}\vee\frac{p-1}{p} \leq r(\tau)\alpha^{-1} \\
& \Longleftrightarrow \{(1-\alpha^{-1})(1-\tau^2)^2+2\tau^2\}\vee\frac{p-1}{p} \leq (p-1)\alpha^{-1}
\end{align*}
where $r(\tau) = p-1.$
Then we have $\alpha_G(p,1,\tau)$ in~\eqref{eq:alphaG1:1}-\eqref{eq:alphaG1:2}.
The maximum of $\alpha_G(p,1,\tau)$ over $\tau$ is attained at $\tau = 0$ and $\alpha_G(p,1) = p.$
\end{proof}
For $d=1$, Theorem~\ref{thm:main2} reproduces the remarkable result in~\cite{EW24}, and also~\cite{M25}. 
For $d\geq 2$, $\alpha_G(p,d)$ does not have an explicit form. We briefly discuss how this search for $\alpha_G(p,d,\tau)$ works. We take~\eqref{eq:alphaG3:1}-\eqref{eq:alphaG3:2} as an example: for any $\tau$ with $h(\tau)\leq 1$, we first solve a quadratic equation for the larger $\alpha_*$ satisfying:
\[
\frac{r(\tau)}{\alpha} = 2(q(\tau)+d\tau^2) - \frac{1}{(1-\alpha^{-1})d}.
\]
Then we check if $\alpha_*$ satisfies the constraints $1 - \alpha_*^{-1} \geq (dq(\tau)h(\tau))^{-1}$. If it satisfies, then we set $\alpha_G(p,d,\tau) = \alpha_*$; otherwise, we go to the next case. We later show that $\alpha_G(p,d)$ is guaranteed to attain in the case~\eqref{eq:alphaG2:1} or~\eqref{eq:alphaG3:1} in Theorem~\ref{thm:main2-sim2} if $p$ is slightly larger than $d$.
This procedure does not yield a simple closed form but we can easily compute it numerically, as shown in the table.

\begin{table}[h!]
\centering
\caption{Numerical values of $\alpha_G(p,d)$ for different $p$ and $d$}
\begin{tabular}{c|cccccccccc}
\hline
$d \backslash p$ 
& 2 & 3 & 4 & 5 & 6 & 7 & 8 & 9 & 10 & 11 \\
\hline
1 & 2 & 3 & 4 & 5 & 6 & 7 & 8 & 9 & 10 & 11 \\
2 & $\times$ & 1.1141 & 1.4831 & 1.8465 & 2.2071 & 2.5736 & 2.9438 & 3.3216 & 3.7051 & 4.0925 \\
3 & $\times$ & $\times$ & 1.0103 & 1.2309 & 1.4721 & 1.6946 & 1.9218 & 2.1605 & 2.4087 & 2.6632 \\
4 & $\times$ & $\times$ & $\times$ & $<1$ & 1.1447 & 1.3177 & 1.4778 & 1.6414 & 1.8161 & 2.0000 \\
5 & $\times$ & $\times$ & $\times$& $\times$ & $<1$& 1.1025 & 1.2355 & 1.3600 & 1.4869 & 1.6227 \\
\hline
\end{tabular}
\end{table}

We also provide a comparison of $\alpha_G(p,d)$ with $\alpha_M(p,d)$ in~\eqref{def:alphaM} and $\alpha_G(p,d,1)$ in~\cite[Theorem 3]{L25} in Figure~\ref{fig:comparison_v0} which shows a clear improvement $\alpha_G(p,d) \geq \alpha_M(p,d) > \alpha_G(p,d,1)$. We also compare Theorem~\ref{thm:main2} with the existing literature in the table below.

\begin{figure}[h!]
\centering
\includegraphics[width=100mm]{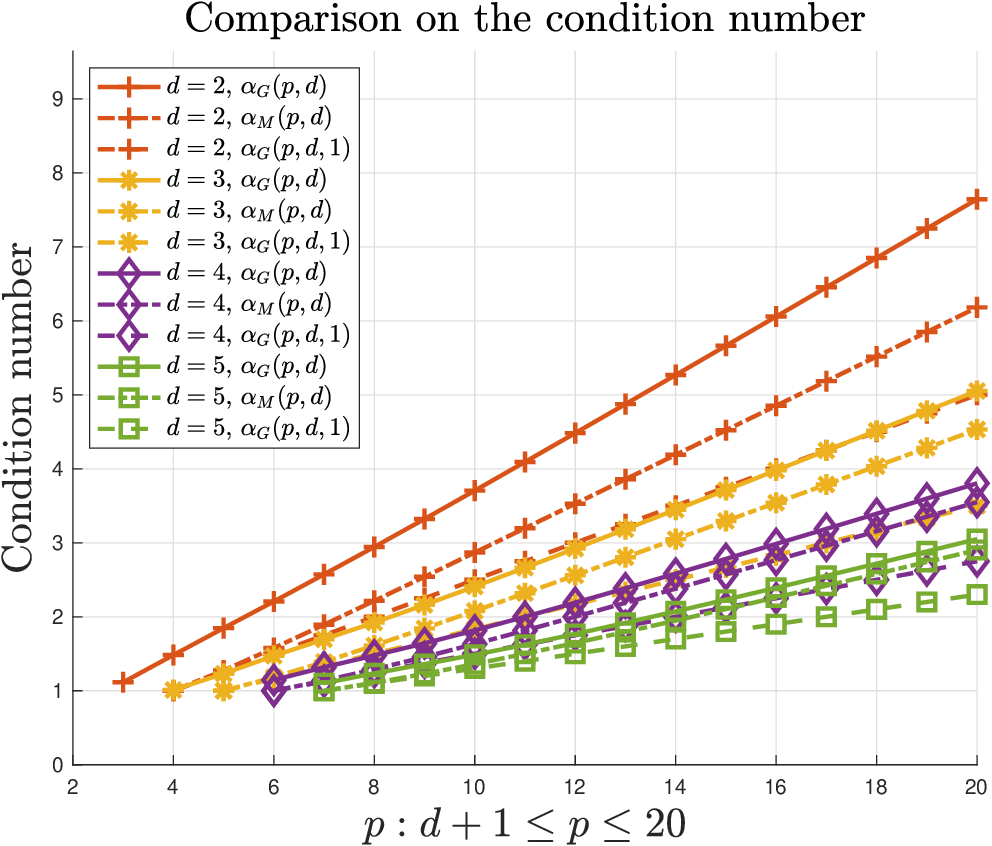}
\caption{Solid line: our bound $\alpha_G(p,d)$ in~\eqref{def:alpha}; dashed dot line: $\alpha_M(p,d)$; dashed line: the state-of-the-art bound $\alpha_G(p,d,1)$. The current bound $\alpha(p,d)$ improves the ones in the existing works.}
\label{fig:comparison_v0}
\end{figure}

\begin{table}[h!]
\centering
\caption{Comparison with the existing works}
\begin{tabular}{Sc|Sc|Sc}
\hline
Reference & $d = 1$ & $d \geq 2 $ \\
\hline 
\cite{L25} 
& $\dfrac{\lambda_{\max}(\BL)}{\lambda_2(\BL)} \leq \dfrac{p-1}{2}$, $p\geq 4$ 
& $\dfrac{\lambda_{\max}(\BL)}{\lambda_{d+1}(\BL)} \leq \dfrac{p+d-2}{2d}$, $p \geq d + 3$ \\
\hline
\cite{EW24,M25} 
& $\dfrac{\lambda_{\max}(\BL)}{\lambda_2(\BL)} < p$, $p\geq 2$ 
& Not covered\\
\hline
This work 
& $\dfrac{\lambda_{\max}(\BL)}{\lambda_2(\BL)} < p$, $p\geq 2$ 
& \makecell[l]{%
$\dfrac{\lambda_{\max}(\BL)}{\lambda_{d+1}(\BL)} < \alpha_M(p,d) < \alpha_G(p,d)$  \\
where $\alpha_G(p,d)$ and $\alpha_M(p,d)$ are in~\eqref{def:alpha} and~\eqref{def:alphaM}. \\
(a) $\alpha_G(p,d)>1$ only if $p \geq d+2$ with $d\geq 4 $,  \\
~~~~ and $p \geq d+1$ with $d=2$ and $3$; \\
(b) $\alpha_M(p,d)>1$ only if $p \geq d+3$ with $d\geq 2$; \\
(c) $\alpha_M(p,d) > \alpha_G(p,d,1) = \dfrac{p+d-2}{2d}. $ \\
} \\
\hline
\end{tabular}
\end{table}

Now we provide the examples of $\alpha_G(p,d,\tau)$ for  $\tau = 0$ and 1. In particular, the following corollary implies the author's earlier work~\cite{L25} for $\tau=1$. 


\begin{corollary}
For $\tau = 0$, then
\[
q(0)= d,~~h(0) = \frac{p}{d(p-d)} < 1,~~\forall d\geq 2,~~p\geq d+2.
\]
Then
\begin{align*}
 \alpha_G(p,d,0)  = 
\begin{dcases}
\frac{p+d + \sqrt{(p+d)^2 - 4 (2 d - d^{-1}) (p-d) }}{2 (2d - d^{-1}) }, & \text{if }~~\alpha^{-1} \leq 1 - \frac{p-d}{pd}, \\
\frac{p-d}{2 d -  \frac{p}{p-d} }, & \text{if }~~ \alpha^{-1} \geq 1 - \frac{p-d}{pd},
\end{dcases} 
\end{align*}
where $q(0) = d$ and $r(0) = p-d$.

For $\tau = 1$, then
\[
q(1)= 0,~~\lim_{\tau\rightarrow 1^-}\frac{1}{q(\tau)} = +\infty,~~ \lim_{\tau\rightarrow 1^-}h(\tau) = +\infty.
\]
Therefore, we have
\[
\alpha_G(p,d,1) = \frac{p+d-2}{2d},~~~ -\infty <\alpha^{-1} < \infty,
\]
where  $r(1) = p+d-2$. In particular, $\alpha_G(p,d,1)=\alpha_M(p,d,1)$ holds.
\end{corollary}

Since $\alpha_G(p,d)$ is hard to visualize, the theorem below provides an approximate and simple lower bound of $\alpha_G(p,d)$. 

\begin{theorem}[\bf Simplification of Theorem~\ref{thm:main2} for $p\geq d+2+O(\sqrt{d})$]\label{thm:main2-sim2}
Suppose
\begin{equation}\label{def:pthm}
p \geq \left( 1 + \dfrac{2\sqrt{d}}{d-1}\right) d +2, ~~\forall d\geq 2,
\end{equation}
then it holds
\begin{align*}
\alpha_G(p,d) 
&  \geq  \max_{0\leq \tau\leq 1} \frac{r(\tau) + d^{-1}}{2(q(\tau)+ d\tau^2) - d^{-1}} = \frac{\tau_*}{2\tau_*-1}
\end{align*}
where the maximizer $\tau_*$ is the larger root of the quadratic equation
\[
 \tau^2 - \frac{(4d-2p-3d^{-1})\tau}{2(d-1)} -  \frac{p-d+d^{-1}}{2(d-1)} = 0.
\]
In other words, under~\eqref{def:pthm}, $\alpha_G(p,d)$ is attained in the case~\eqref{eq:alphaG2:1} or~\eqref{eq:alphaG3:1}.
\end{theorem}

\subsection{Tightness of the bound and future work}\label{ss:main2}
One natural question regarding Theorem~\ref{thm:main} concerns its optimality: what is the sharp dependence of the benign landscape condition on the condition number? In fact, our proof also yields, as a byproduct, a lower bound on the condition number guaranteeing that a specific choice of 
$\BS$ is not a second-order critical point. By using  this bound, we establish a sufficient condition under which every $\BS\in\St(p,d)^{\otimes n}$ satisfying 
$\BZ^{\top}\BS = 0$ fails to be a second-order critical point. Moreover, we construct a family of “twisted states” showing that this condition is also necessary.

\begin{theorem}\label{thm:optimality}
\begin{enumerate}[(a)]
\item For $\BS\in\St(p,d)^{\otimes n}$ such that $\BZ^{\top}\BS = 0$, if
\[
\frac{\lambda_{\max}(\BL)}{\lambda_{d+1}(\BL)} < \frac{2p}{d+1}
\]
then any $\BS\in\St(p,d)^{\otimes n}$ with $\BZ^{\top}\BS= 0$ is not a second-order critical point. 

\item Conversely, there exists an $\BS\in\St(p,d)^{\otimes n}$ with $\BZ^{\top}\BS=0$ and $\BL\succeq 0$ such that $\BS$ is a second-order critical point coexisting with the global minimizer $\BZ$ to~\eqref{def:bm} and the condition number satisfies
\[
\frac{\lambda_{\max}(\BL)}{\lambda_{d+1}(\BL)} = \frac{2p}{d+1}.
\]
\end{enumerate}
\end{theorem}
The proof is given in Section~\ref{ss:proof:optimality}.
Theorem~\ref{thm:optimality} essentially gives a sufficient and necessary condition such that any  $\BS\in\St(p,d)$ with $\BZ^{\top}\BS = 0$, is not a second-order critical point. In particular, Theorem~\ref{thm:optimality}(b) implies that if the condition number of $\BL$ is larger than $2p/(d+1)$, some counterexamples can be constructed such that a spurious second-order critical point exists. This counterexample was first constructed in~\cite{EW24} for $d=1.$ 
Our Theorem~\ref{thm:optimality}(b) extends this family of counterexamples, known as ``twisted states", from $d=1$ to $d\geq 2$, as shown below.

\begin{example}[\bf Counterexample]
For the proof of Theorem~\ref{thm:optimality}(b), we construct examples of $\BL$  such that 
\[
\{\BS_i\}_{i=1}^{n}= \left\{  \BD_j \BQ_k
\right\}_{1\leq j\leq 2^d,1\leq k\leq p}
\]
is a second-order critical point where $n = 2^dp$, 
\[
\BD_j \in \left\{
\diag(\{\pm 1\}^d)
\right\},~1\leq j\leq 2^d
\]
is a family of diagonal matrices of size $d\times d$ with diagonal entries $\{\pm 1\}$  and
\[
\BQ_k = \{ [\I_d, 0_{d\times (p-d)}]\BM^k:1\leq k\leq p \},~~~
\BM = 
\begin{bmatrix}
0 & 1 & 0 & \cdots & 0 & 0 \\
0 & 0 & 1 & \cdots & 0 & 0 \\
\vdots & \vdots & \vdots & \ddots & \vdots & \vdots  \\
0 & 0 & 0 & \cdots & 0 & 1 \\
1 & 0 & 0 & \cdots & 0 & 0 
\end{bmatrix}\in\RR^{p\times p}.
\]
This construction of $\{\BS_i\}$ satisfies
\[
\BS^{\top}\BS = \frac{nd}{p}\I_p,~~\BZ^{\top}\BS = 0,
\]
where $\BZ^{\top} = [\I_d,\cdots,\I_d].$
In this case, we can show that if 
\[
\BL = \frac{2p}{d+1} \left(\I_{nd} - n^{-1}\BZ\BZ^{\top}\right) - \left(\frac{2p}{d+1}-1\right)\frac{p}{nd}\BS\BS^{\top}
\]
then $\BS$ is a second-order critical point coexisting with the global minimizer $\BZ$ to~\eqref{def:bm}. For this $\BL$, we can show that 
\[
\lambda_{\max}(\BL) = \frac{2p}{d+1},~~\lambda_{d+1}(\BL)=1.
\]
\end{example}

The state-of-the-art bound $\alpha_G(p,d)$ differs from $2p/(d+1)$ by a factor of roughly 4. 
Given that this type of ``twisted state" achieves the optimal bound for $d=1$, we make a conjecture. 
\begin{conjecture}
The optimization landscape of~\eqref{def:bm} is benign  if
\[
\frac{\lambda_{\max}(\BL)}{\lambda_{d+1}(\BL)} < \frac{2p}{d+1}
\]
where $\BL\BZ= 0 $ and $\BL\succeq 0.$
\end{conjecture}
We leave this intriguing question for future work. The main technical bottleneck lies in the construction of the dual variables in Lemma~\ref{lem:certificate}, where the current certificate is not sharp. Improving this step seems to require a refined structural characterization of the extreme configurations in the example above.

\subsection{Applications}\label{ss:main3}

Theorem~\ref{thm:main} improves the bound on the condition number for $d\geq 2$ in~\cite{L25}, and matches those for $d=1$ in~\cite{EW24,M25}, and as an immediate result, it is straightforward to argue that the current result also improves all the downstream applications in various applications. For the completeness of the presentation,  we elaborate some applications of the main theorem. The key is to control the condition number of $\BL$ such that
$\lambda_{\max}(\BL)/\lambda_{d+1}(\BL) < \alpha_G(p,d)$
guarantees a benign landscape. 

\paragraph{Application 1: Kuramoto model and $\mathbb{Z}_2$-synchronization} Suppose $d=1$ and $p=2$, then it is the synchronization problem on $\mathbb{S}^1$, see~\cite{LXB19,EW24,ABK22,M25,L25} for more details. The gradient system associated with corresponding objective function yields the well-known Kuramoto model~\cite{K75}.   More precisely, let $\bs_i\in\mathbb{S}^1$ with the parameterization $\bs_i = (\cos\theta_i,\sin\theta_i)$, and then the Kuramoto model is equal to the gradient system of the following energy function under the trignometric parameterization:
\[
\frac{\diff \btheta}{\diff t} = - \nabla_{\btheta}E(\btheta),~~~E(\btheta) = -\sum_{i<j} a_{ij}\cos(\theta_i-\theta_j),
\]
where $E(\btheta)$ is exactly equal to $-\lag \BA, \BS\BS^{\top}\rag$  modulo a constant with $\BS\in (\mathbb{S}^1)^{\otimes n}.$ The core problem is to understand whether the oscillators will reach global synchrony instead of getting stuck at non-synchronized stable equilibria which correspond to the local minima of the nonconvex energy function $E(\btheta).$

Theorem~\ref{thm:main} provides a characterization of the global Kuramoto synchrony that only depends on the condition number of the certificate matrix (Laplacian) $\BL$ that matches the results in~\cite{EW24,M25}. To be more precise, for the adjacency matrix $\BA = (a_{ij})_{i,j}$ where $a_{ij}$ could be positive or negative, then $\BJ_n$ is the global minimizer to the SDR~\eqref{def:sdr} if the graph Laplacian satisfies $\BL = \BD - \BA\succeq 0$ where $\BD = \ddiag(\BA\BJ_n)$ is the degree matrix. 
Then the theorem implies that as long as the graph Laplacian associated with the adjacency matrix $\BA$ is well-conditioned, i.e.,
\[
\frac{\lambda_{\max}(\BL)}{\lambda_2(\BL)}< 2,
\]
then the landscape is benign. This applies to the Kuramoto model on Erd\'os-Reny\'i random networks and signed network, which is closely related to the $\mathbb{Z}_2$-synchronization. We do not repeat the same argument here as it has been well discussed in~\cite{EW24,M25,L25} for $d=1$ and $p=2.$


\paragraph{Application 2: orthogonal group synchronization}

Consider the orthogonal synchronization under additive Gaussian noise, i.e., see~\cite{MMMO17,L22}, 
\[
\BA_{ij} = \BO_i\BO_j^{\top } + \sigma\BW_{ij}
\]
where $\BW_{ij}$ is a $d\times d$ Gaussian random matrix. It is known from~\cite[Theorem 3.1]{L22} that the SDR~\eqref{def:sdr} is tight with high probability provided that $\sigma\lesssim \sqrt{n}/(\sqrt{d}(\sqrt{d} + \sqrt{\log n}))$. Moreover, based on the discussion in~\cite[Proof of Corollary 6 in Section 3.2]{L25}, we know that
\[
\frac{\lambda_{\max}(\BL)}{\lambda_{d+1}(\BL)} \leq \frac{n + C\sigma\sqrt{nd}(\sqrt{d} + 4\sqrt{\log n})}{n - C\sigma\sqrt{nd}(\sqrt{d} + 4\sqrt{\log n})} 
\]
with some absolute constant $C> 0$, and $\BL = \BDG(\BA\BZ\BZ^{\top})- \BA\succeq 0$ and $\BZ\BZ^{\top}$ is the unique global minimizer to~\eqref{def:sdr}.

Combining the estimation of $\lambda_{\max}(\BL)/\lambda_{d+1}(\BL)$ with Theorem~\ref{thm:main}, we have the following result regarding the global benign landscape. 
\begin{corollary}
For the orthogonal synchronization under additive Gaussian noise, then the optimization landscape of~\eqref{def:bm} is benign with high probability for any $p\geq d+2$ provided that
\[
\sigma \leq \frac{\alpha_G(p,d)-1}{\alpha_G(p,d)+1} \frac{\sqrt{n}}{C\sqrt{d}(\sqrt{d} + 4\sqrt{\log n})}.
\]
In particular, for $p\geq d+3$, then
\[
\sigma \leq \frac{\alpha_M(p,d)-1}{\alpha_M(p,d)+1}
\cdot  \dfrac{\sqrt{n}}{C\sqrt{d}(\sqrt{d} + 4\sqrt{\log n})} = \frac{1-\tau_*}{3\tau_*-1}
\cdot  \dfrac{\sqrt{n}}{C\sqrt{d}(\sqrt{d} + 4\sqrt{\log n})}
\] 
suffices where $\tau_*\in (1/2,1)$ is given in~\eqref{def:alphaM} and $\alpha_M(p,d) = \tau_*/(2\tau_*-1).$
\end{corollary}
This bound improves~\cite[Corollary 6]{L25} by a constant and also extends the range of $p$ from $p\geq d+3$ to any $p\geq d+2$. It also implies a benign landscape of the Burer-Monteiro factorization with $p$ slightly larger than $d$ under nearly identical conditions that guarantee the tightness of the SDR.

\paragraph{Application 3: generalized orthogonal Procrustes problem}
Consider the generalized orthogonal Procrustes problem with the following statistical model:
\[
\BA_i = \BO_i\bar{\BA} + \sigma\BW_i,~~~1\leq i\leq n,
\]
where $\bar{\BA}\in\RR^{d\times m}$ and $\BW_i\in\RR^{d\times m}$ is a Gaussian random matrix, and $d\geq 2$. The generalized orthogonal Procrustes problem aims to recover $\{\BO_i\}$ and $\bar{\BA}$ from $\{\BA_i\}$~\cite{SS11,BKS14,L23b}, i.e., find an orthogonal transform for each point cloud $\BA_i$ with $m$ points in $\RR^d$ such that the transformed datasets are well aligned. The least squares estimation equals 
\[
\max_{\BG_i\in \Od(d)}~\sum_{i<j} \lag \BA_i\BA_j^{\top}, \BG_i \BG_j^{\top}\rag,
\]
which resembles the optimization problem in~\eqref{def:pp}. Similar to the orthogonal group synchronization, one can find a convex relaxation in the form of~\eqref{def:sdr} for~\eqref{def:pp}, and also consider the Burer-Monteiro factorization~\eqref{def:bm} and its global optimization landscape.
The author's earlier work~\cite{L23b} proves that if 
\[
\sigma \lesssim \frac{\sigma_{\min}(\bar{\BA})}{\kappa^4}\cdot \frac{\sqrt{n}}{\sqrt{d}(\sqrt{nd} + \sqrt{m} + 2\sqrt{n\log n})},
\]
then the SDR is tight, i.e.,~\eqref{def:sdr} produces a rank-$d$ solution that exactly equals the least squares estimator.

Regarding the global optimization landscape of the Burer-Monteiro factorization, we note that~\cite[Corollary 7]{L25} implies an upper bound on the condition number of the Hessian at the global minimizer. Combining it with Theorem~\ref{thm:main}, we know that the landscape is benign if
\[
\frac{\lambda_{\max}(\BL)}{\lambda_{d+1}(\BL)} \leq \kappa^2\cdot \frac{\sqrt{n}\|\bar{\BA}\| +C \kappa^4 \sigma \sqrt{d} (\sqrt{nd} + \sqrt{m} + \sqrt{2\gamma n\log n}) }{ \sqrt{n}\|\bar{\BA}\| -  C \kappa^4 \sigma \sqrt{d} (\sqrt{nd} + \sqrt{m} + \sqrt{2\gamma n\log n})} \leq \alpha_G(p,d).
\]

Therefore, we write the discussion above into a corollary which is clearly an improved result from the author's previous work~\cite{L25} because of $\alpha_G(p,d) > \alpha_G(p,d,1)$.
\begin{corollary}
For the generalized orthogonal Procrustes problem under additive Gaussian noise, then the optimization landscape of~\eqref{def:bm} is benign with high probability provided that
\[
\sigma \leq 
\frac{\alpha_G(p,d) - \kappa^2}{\alpha_G(p,d)+\kappa^2}\cdot  \dfrac{\sqrt{n}\|\bar{\BA}\|}{ C \kappa^4 \sqrt{d} (\sqrt{nd} + \sqrt{m} + \sqrt{2\gamma n\log n})}
\]
for  $p \geq d+2$.
\end{corollary}

\subsection{Roadmap of the proof}\label{ss:roadmap}
In this section, we provide a roadmap of the proof, i.e., only providing the main lemmas/propositions without giving the technical details. The details are deferred to Section~\ref{s:proof}.
The proof idea is inspired by that in~\cite{EW24} for $d=1$ by Rakoto Endor and Waldspurger. 
To show a benign optimization landscape, it suffices to prove its contrapositive statement: suppose $\BZ\BZ^{\top}$ is the unique global minimizer to the SDR~\eqref{def:sdr} and there exists a non-global second-order critical point in~\eqref{def:bm}, then the condition number of the certificate matrix (Laplacian) at the global minimizer $\BZ\BZ^{\top}$ has a nontrivial lower bound. A lower bound was first established in~\cite{L25}. Our goal is to obtain a sharper bound by relaxing the problem to a convex optimization program and then constructing a feasible solution to its dual (a dual certificate), which in turn yields an improved lower bound. Unlike $d=1$ in~\cite{EW24}, the construction of the dual variables which is used to provide a lower bound of the condition number is much more complicated for $d\geq 2$. 
Therefore, our major contribution is to extend the proof from $d=1$ to $d\geq 2$. This extension requires more delicate treatments and constructions because $\Od(d)$ becomes non-commutative for $d\geq 2$. 
Before we proceed to explain the main technical steps, we assume without loss of generality that  $\BZ_i = \I_d$
 and we define
\[
\BZ^{\top} : = [\I_d,\cdots,\I_d]\in\RR^{d\times nd}.
\]
This is because we can replace $\BL$ by $\blkdiag(\BZ_1,\cdots,\BZ_n)\BL \blkdiag(\BZ_1,\cdots,\BZ_n)$. 

\paragraph{Convex relaxation of the condition number minimization.}
The high-level illustration of the optimization problem is
\begin{equation}\label{def:highprog}
\begin{aligned}
\min_{\BL}~~ & \frac{\lambda_{\max}(\BL)}{\lambda_{d+1}(\BL)} \\
\text{s.t.}~~& \BZ\BZ^{\top} \text{ is the global minimizer to~\eqref{def:sdr} with the Hessian }\BL, \\
& \BS \text{ is a second-order critical point of~\eqref{def:bm}}.
\end{aligned}
\end{equation}

Surprisingly, this optimization problem can be formulated into a convex problem, and its lower bound can be provided by evaluating the corresponding dual program. More precisely, the problem~\eqref{def:highprog} can be made into 
\begin{equation}\label{prog:original}
\begin{aligned}
\min_{\BL} \quad &\frac{\lambda_{\max}(\BL)}{\lambda_{d+1}(\BL)} \\
\text{s.t.}\quad & \BL  \succeq 0, \\
& \BL\BZ = 0, \\
& \lambda_{d+1}(\BL) > 0, \\
& (\BL - \BDG(\BL\BS\BS^{\top})) \BS = 0, \\
& \lag \BL - \BDG(\BL\BS\BS^{\top}), \dot{\BS}\dot{\BS}^{\top}\rag \geq 0,
\end{aligned}
\end{equation}
where $\dot{\BS}_i$ is on the tangent space of $\St(p,d)$ at $\BS_i$ and $\dot{\BS}\in\RR^{nd\times p}$ consists of $\{\dot{\BS}_i\}_{i=1}^n.$
The first three constraints focus on all the Laplacian matrices $\BL$ that ensure the global optimality of $\BZ\BZ^{\top}$ to 
\[
\min_{\BX\succeq 0,\BX_{ii}=\I_d}~\lag\BL, \BX \rag
\] 
i.e., an equivalent form of~\eqref{def:sdr} because the diagonal blocks remain constant.
The last two constraints ensure that $\BS\in\St(p,d)^{\otimes n}$ is a second-order critical point to its Burer-Monteiro factorization~\eqref{def:bm}.
In fact, we have the following proposition that shows ~\eqref{def:highprog} and~\eqref{prog:original} are equivalent.
\begin{proposition}\label{prop:equivalence1}
The problem~\eqref{def:highprog} is equivalent to~\eqref{prog:original}.
\end{proposition}
We proceed to solve~\eqref{prog:original} by finding  a convex relaxation. 
Define
\[
\BP_{\perp} := \I_{nd} - \frac{1}{n}\BZ\BZ^{\top}
\]
and $\BL\BZ = 0$ implies that $\BL = \BP_{\perp}\BL\BP_{\perp}.$ Also we denote
\[
\BY_{\perp} := \BP_{\perp}\BY\BP_{\perp},~~\forall \BY\in\RR^{nd\times nd}.
\]
The convex relaxation of~\eqref{prog:original} and its dual form are written into the following proposition.

\begin{proposition}\label{prop:convexrelax}
The optimization problem~\eqref{prog:original} has a convex relaxation of the following form:
\begin{equation}\label{def:primal}
\begin{aligned}
\min_{\BL,t}\quad &~ t \\
m\geq 0 :&~ t \geq 0 \\
\BX\succeq 0:&~ \BL_{\perp} \preceq t\BP_{\perp}, \\
\BQ\succeq 0:&~ \BL_{\perp} \succeq \BP_{\perp}, \\
\BB\in \RR^{nd\times p}:& ~(\BL_{\perp} - \BDG(\BL_{\perp}\BS\BS^{\top})) \BS = 0, \\
\BT\in K^{**}:&~ \BL_{\perp} -\BDG(\BL_{\perp}\BS\BS^{\top}) \in K^*.
\end{aligned}
\end{equation}
Here $K^{**}$ is the dual cone of $K^*$, which is the smallest convex cone containing 
\[
K:=\{\dot{\BS}\dot{\BS}^{\top}: \BS_i\dot{\BS}_i^{\top}+\dot{\BS}_i\BS_i^{\top} = 0\}.
\] In addition, $m$, $\BX$, $\BQ$, $\BB$, and $\BT$ are the corresponding dual variables associated with the constraints.

\begin{equation}\label{prog:dual_reduce}
\begin{aligned}
\max_{\BX,\BQ,\BB,\BT} \quad &  \lag \BP_{\perp}, \BQ\rag \\
{\rm s.t.}~~& \BX\succeq 0,~~\BQ\succeq 0, \\
& \lag\BP_{\perp}, \BX\rag \leq 1, \\
& \BP_{\perp}\left(\BX - \BQ + \frac{1}{2}(\BB\BS^{\top}+\BS\BB^{\top}) - \BT\right)\BP_{\perp} \\
&\qquad + \frac{1}{2}\BP_{\perp}\left( \BDG(\BT - \BB\BS^{\top})\BS\BS^{\top} + \BS\BS^{\top} \BDG(\BT - \BB\BS^{\top})\right)\BP_{\perp} = 0, \\
& \BT \in K^{**}.
\end{aligned}
\end{equation}
The strategy is to find a lower bound of the primal program by constructing a dual variable that achieves high dual value.
\end{proposition}
We find that~\eqref{def:primal} does not have a feasible solution for every $\BS$, i.e., some $\BS\in\St(p,d)^{\otimes n}$ cannot be a second-order critical point that co-exists with the global minimizer $\BZ$. In that case, the value to~\eqref{prog:original} is $+\infty.$
Note that the earlier work~\cite{L25} proves that
\[
\frac{\lambda_{\max}(\BL)}{\lambda_{d+1}(\BL)} \leq \frac{p+d-2}{2d}, 
\]
ensures the benignness of the optimization landscape. Its equivalent statement is: if there exists a spurious second-order critical point  $\BS\in\St(p,d)^{\otimes}$ that is not global, then the condition number of the Laplacian at the global minimizer, i.e., the global minimizer to~\eqref{prog:original} is at least $(p+d-2)/2d$. It also means that the optimal value of~\eqref{prog:dual_reduce} is at least 
\[
\frac{p+d-2}{2d},~~\forall \BS\in\St(p,d)^{\otimes n}.
\]
Therefore, the main goal is to find a refined uniform lower bound for the global minimizer to~\eqref{prog:original} for all $\BS\in\St(p,d)^{\otimes n}.$

\paragraph{Construction of the dual variables}
The idea of constructing a feasible dual variable in~\eqref{prog:dual_reduce} is to first simplify some constraints. Unlike the earlier works~\cite{EW24,L25,MB24} in which a specific choice of $\dot{\BS}_i$ is chosen, we will consider a family of tangent vectors and try to find the optimal one for any given pair of $(p,d).$ We first introduce this family of tangent vectors, and present a few important quantities that will be used in the construction of feasible dual variables. 

\begin{definition}
We define two linear operators $\RR^{d\times p}$ to $\RR^{d\times p}$:
 \begin{equation}\label{def:LP}
\begin{aligned}
{\cal L}_{\BS_i,\tau}(\BY) & : =  \BY\BS_i^{\top}\BS_i + \tau(  \BS_i\BY^{\top} -\BY \BS_i^{\top})\BS_i
  = (1-\tau)\BY\BS_i^{\top}\BS_i + \tau\BS_i\BY^{\top}\BS_i, \\
{\cal P}_{\BS_i,\tau}(\BY) & := \BY - {\cal L}_{\BS_i,\tau}(\BY) = \BY(\I_p - \BS_i^{\top}\BS_i) - \tau(  \BS_i\BY^{\top} -\BY \BS_i^{\top})\BS_i,
\end{aligned}
\end{equation}
where $\tau$ is a parameter and $\BS_i\in\St(p,d).$
\end{definition}
Apparently, ${\cal P}_{\BS_i,\tau}(\BY)$ is on the tangent space that satisfies
\[
{\cal P}_{\BS_i,\tau}(\BY)\BS_i^{\top} + \BS_i {\cal P}_{\BS_i,\tau}(\BY)^{\top} = 0,~~\forall \BY\in\RR^{d\times p}.
\]
In particular, for $\tau = 1/2$, ${\cal P}_{\BS_i, 1/2}(\cdot)$ is the projection onto the tangent space of $\St(p,d)$ at $\BS_i.$ In the author's earlier work, only $\tau = 1$  is considered, i.e., ${\cal P}_{\BS_i,1}(\BY) = \BY - \BS_i\BY^{\top}\BS_i$. Therefore, this family of tangent vectors is much more general than those in the existing works. 
In addition, we can see that both ${\cal L}_{\BS_i,\tau}(\BY)$ and ${\cal P}_{\BS_i,\tau}(\BY)$ contain anti-symmetric terms $\BS_i\BY^{\top}-\BY\BS_i^{\top}$ for $d\geq 2$. For $d=1$, the anti-symmetric term is gone, i.e., ${\cal L}_{\BS_i,\tau}$ and ${\cal P}_{\BS_i,\tau}$ remain constant for any $\tau.$

In our analysis, we need to apply ${\cal L}_{\BS_i,\tau}$ and ${\cal P}_{\BS_i,\tau}$ for $1\leq i\leq n$  jointly. Thus we also define
\begin{equation}\label{def:LPn}
\begin{aligned}
{\cal L}_{\BS,\tau}(\BY) & := 
\begin{bmatrix}
{\cal L}_{\BS_1,\tau}(\BY) \\
\vdots \\
{\cal L}_{\BS_n,\tau}(\BY)
\end{bmatrix}\in\RR^{nd\times p},~~~~
{\cal P}_{\BS,\tau}(\BY) & := \BZ\BY - {\cal L}_{\BS,\tau}(\BY)
\end{aligned}.
\end{equation}
The properties of ${\cal L}_{\BS_i,\tau}(\cdot)$,  ${\cal L}_{\BS,\tau}(\cdot)$,  ${\cal P}_{\BS_i,\tau}(\cdot)$, and ${\cal P}_{\BS,\tau}(\cdot)$ are discussed in Section~\ref{ss:proof_2}. 
In particular, we define
\[
\BSigma_{\BS,\tau} := \E {\cal L}_{\BS,\tau}(\BPhi){\cal L}_{\BS,\tau}(\BPhi)^{\top},~~~\E {\cal P}_{\BS,\tau}(\BPhi){\cal P}_{\BS,\tau}(\BPhi)^{\top} = (2(d-1)\tau+p-2d)\BZ\BZ^{\top} + \BSigma_{\BS,\tau}
\]
as the ``covariance" of ${\cal L}_{\BS,\tau}(\BPhi)$ and ${\cal P}_{\BS,\tau}(\BPhi)$ where $\BPhi\in\RR^{d\times p}$ is a Gaussian random matrix with i.i.d. entries. In particular, it holds
\begin{equation}\label{eq:lptrace}
(\BSigma_{\BS,\tau})_{ii} = (q(\tau) + d\tau^2)\I_d,~~~\E ({\cal P}_{\BS,\tau}(\BPhi){\cal P}_{\BS,\tau}(\BPhi)^{\top})_{ii} = r(\tau)\I_d.
\end{equation}

We are now ready to give our construction of the dual variables as follows.
\begin{lemma}\label{lem:certificate}
Let $\BQ$, $\BT$ and $\BB$ in~\eqref{prog:dual_reduce} be
\begin{align*}
\BQ & =\frac{\alpha\BS\BS^{\top}}{\lag \BP_{\perp}, \BS\BS^{\top}\rag}, \\
\BT & = \beta \E {\cal P}_{\BS,\tau}(\BPhi){\cal P}_{\BS,\tau}(\BPhi)^{\top}=   \beta \left( (2(d - 1)\tau+ p - 2d
)\BZ\BZ^{\top} + \BSigma_{\BS,\tau}\right), \\ 
 \BB & = \delta \BC ,~~~\BC = \PP_{\BS,\tau}(\BY),~~\|\BY\|_F= 1, 
\end{align*}
for some $\beta>0,\delta\in\RR$, and $\BY\in\RR^{d\times p}$. Then it holds that
\[
\BT\in K^{**},~~\BT_{ii}  = \beta r(\tau)\I_d,~1\leq i\leq n,
\]
and
\begin{align*}
\BX_{\perp} & = \left(  \frac{\alpha}{\lag \BP_{\perp}, \BS\BS^{\top}\rag} - \beta r(\tau) \right) (\BS\BS^{\top})_{\perp}  - \frac{\delta}{2}(\BC\BS^{\top} + \BS\BC^{\top})_{\perp}  + \beta \left( \BSigma_{\BS,\tau}\right)_{\perp} \succeq 0
\end{align*}
holds if
\[
\begin{bmatrix}
\dfrac{\alpha}{\lag \BP_{\perp}, \BS\BS^{\top}\rag} - \beta r(\tau)  & -\dfrac{\delta}{2} \\
-\dfrac{\delta}{2} & \beta
\end{bmatrix}\succeq 0
\]
where $\BC_i\in\RR^{d\times p}$ is the $i$-th block of $\BC\in\RR^{nd\times p}$ and it holds that 
\begin{equation}\label{eq:key_psd}
\left(\BSigma_{\BS,\tau} \right)_{\perp} \succeq (\BC\BC^{\top})_{\perp} = ({\cal L}_{\BS,\tau}(\BY){\cal L}_{\BS,\tau}(\BY)^{\top})_{\perp}.
\end{equation}

\end{lemma}

By plugging the choice of $\BC_i = \PP_{\BS_i,\tau}(\BY) $ in Lemma~\ref{lem:certificate}, we future simplify the dual program. It is easy to see that the optimal value to the dual program~\eqref{prog:dual_reduce} has a lower bound:
\begin{equation}\label{prog:dual2}
\begin{aligned}
\max_{\alpha,\beta,\delta,\BY}\quad & \alpha \\
\text{s.t.}\quad& \BX_{\perp} := \left(  \frac{\alpha}{\lag \BP_{\perp}, \BS\BS^{\top}\rag} - \beta r(\tau)\right) (\BS\BS^{\top})_{\perp}   - \frac{\delta}{2}(\BC\BS^{\top} + \BS\BC^{\top})_{\perp}  + \beta \left( \BSigma_{\BS,\tau}\right)_{\perp}\succeq 0, \\
& \BC =  {\cal P}_{\BS,\tau}(\BY),~~\|\BY\|_F = 1, \\
& \begin{bmatrix}
\dfrac{\alpha}{\lag \BP_{\perp}, \BS\BS^{\top}\rag} -  \beta r(\tau)  & -\dfrac{\delta}{2} \\
-\dfrac{\delta}{2} & \beta
\end{bmatrix}\succeq 0, \\
& \Tr(\BX_{\perp})\leq 1.
\end{aligned}
\end{equation}

To solve~\eqref{prog:dual2}, we choose
\[
\BY = -\frac{\BZ^{\top}\PP_{\BS,\tau}(\BZ^{\top}\BS)}{\|\BZ^{\top}\PP_{\BS,\tau}(\BZ^{\top}\BS)\|_F}.
\]
Then the SDP~\eqref{prog:dual2} only involves $\alpha,\beta$, and $\delta$.
We want to find a universal $\alpha$ such that for the given $\alpha$, there exists $\delta$ and $\beta$ that depend on $\BS$ such that the constraints in~\eqref{prog:dual2} are feasible for any $\BS\in\St(p,d)^{\otimes n}$.

\begin{lemma}
\label{lem:alpha_find}
The optimal value of $\alpha$ to~\eqref{prog:dual2} is lower bounded by 
\begin{equation}\label{eq:prog_mx0}
\max~\alpha \quad {\rm s.t.}\quad
\alpha^{-1} \geq \frac{ g(\BS,1-\alpha^{-1},\tau)}{ r(\tau)}
\end{equation}
where 
\begin{equation}\label{def:gsx}
\begin{aligned}
 & g(\BS,x,\tau) := \frac{ \Tr(\BSigma_{\BS,\tau})_{\perp} }{\lag \BP_{\perp}, \BS\BS^{\top}\rag} - \frac{n^{-2} \left\| \BZ^{\top}\PP_{\BS,\tau}(\BZ^{\top}\BS) \right\|_F ^2}{ \lag \BP_{\perp}, \BS\BS^{\top}\rag^2 x }
\end{aligned}
\end{equation}
for any fixed $\BS.$
That's to say: if $\lambda_{\max}(\BL)/\lambda_{d+1}(\BL) $ is smaller than the optimal value to~\eqref{eq:prog_mx0}, then any fixed $\BS$ is not a second-order critical point. 
To ensure any $\BS\in\St(p,d)^{\otimes n}$ with $\BS\BS^{\top}\neq\BZ\BZ^{\top}$ is not a second-order critical point, it suffices to solve
\begin{equation}\label{eq:prog_mx00}
\max~\alpha \quad {\rm s.t.}\quad
\alpha^{-1} \geq \frac{ g(\BS,1-\alpha^{-1},\tau)}{ r(\tau)},~~\forall \BS\in\St(p,d)^{\otimes n},~\BS\BS^{\top}\neq\BZ\BZ^{\top}.
\end{equation}

Suppose $M(x,\tau)$ is an upper bound of $\max_{\BS\in\St(p,d)^{\otimes n}} g(\BS,x,\tau)$, i.e, 
\[
\max_{\BS\in\St(p,d)^{\otimes n}} g(\BS,x,\tau) \leq M(x,\tau),
\]
and then it suffices to find
\begin{equation}\label{eq:prog_mx}
\max~\alpha  \quad {\rm s.t.}\quad r(\tau)\alpha^{-1} \geq M(1-\alpha^{-1},\tau) 
\end{equation}
which gives a lower bound to the optimal value of the dual program.
\end{lemma}
For Theorem~\ref{thm:alphaM}, it suffices to find a simple upper bound $M(x,\tau)$ and apply Lemma~\ref{lem:alpha_find}. We can see in~\eqref{def:gsx} that
\[
g(\BS,x,\tau) \leq \frac{\Tr(\BSigma_{\BS,\tau})_{\perp}}{\lag \BP_{\perp}, \BS\BS^{\top}\rag}.
\]
Then using this upper bound of $g(\BS,x,\tau)$ gives Theorem~\ref{thm:alphaM} whose proof is included in Section~\ref{ss:proof_3}.

The final technical step of Theorem~\ref{thm:main} is to find a sharper but still simple upper bound of~\eqref{def:gsx}, which is given in Lemma~\ref{lem:thmmain}:
\begin{equation}\label{eq:gupper}
\sup_{\BS\in\St(p,d)^{\otimes n}} g(\BS,x,\tau)\leq \sup_{0\leq u\leq w\leq v\leq 1}q(\tau) \cdot \frac{ u}{v}\left( 2 - \frac{pu}{p-d}\right) + d\tau^2\cdot\frac{w}{v} \left(2 - \frac{pdw}{pd-1}\right)  - \frac{(1-dv)_+u^2}{x \cdot dv^2}.
\end{equation}
The technical details of Lemma~\ref{lem:alpha_find} and~\ref{lem:thmmain} can be found in Section~\ref{ss:proof_4}-\ref{ss:proof_5}. The exact evaluation of~\eqref{eq:gupper} is given in Proposition~\ref{prop:gfun}.
Finally in Section~\ref{ss:proof_5}, we finish the proof of Theorem~\ref{thm:main} and~\ref{thm:main2} by combining~\eqref{eq:prog_mx} in Lemma~\ref{lem:alpha_find} with~\eqref{eq:gupper} in Lemma~\ref{lem:thmmain}.

\paragraph{A primal argument of the main theorem}
For the completeness of the presentation, 
we complement the contrapositive argument by showing directly that solving the dual program above will lead to a guarantee on the benignness of landscape.
\begin{theorem}[\bf A direct argument for a benign landscape]\label{thm:positive}
Suppose there exists a scalar $\alpha$ that is independent of $\BS$, and $(\beta,\delta,\BY)$ that depends on $\BS$ such that~\eqref{prog:dual2} holds for any $\BS$, i.e.,
\[
\BX_{\perp}\succeq 0,~~\Tr(\BX_{\perp}) \leq 1,
\]
and 
\[
\frac{\lambda_{\max}(\BL)}{\lambda_{d+1}(\BL)} < \alpha.
\]
Then any second-order critical point $\BS$ is a global minimizer satisfying $\BS\BS^{\top} = \BZ\BZ^{\top}.$
\end{theorem}
\begin{proof}
Consider $\BS$ as a second-order critical point that is not equal to $\BZ\BZ^{\top}$:
\[
(\BL - \BDG(\BL\BS\BS^{\top}))\BS =0,~~~\lag \BL - \BDG(\BL\BS\BS^{\top}), \dot{\BS}\dot{\BS}^{\top} \rag\geq 0,~\forall \BS_i\dot{\BS}_i^{\top} + \dot{\BS}_i\BS_i^{\top}=0.
\]
Note that for $\dot{\BS}_i = \PP_{\BS_i,\tau}(\BPhi)$ with $d\times p$ Gaussian random matrix $\BPhi$,  it holds that
\begin{align*}
&  \lag \BL - \BDG(\BL\BS\BS^{\top}), \E\dot{\BS}\dot{\BS}^{\top} \rag \\
& ~~=  \lag \BL - \BDG(\BL\BS\BS^{\top}), \left( 
p - 2((1-\tau)d + \tau) \right)\BZ\BZ^{\top} + \BSigma_{\BS,\tau} \rag \\
& ~~= \left\lag \BL, \BSigma_{\BS,\tau} \right\rag - r(\tau)\left\lag \BL, \BS\BS^{\top}\right\rag \geq 0
\end{align*}
where $\BL\BZ=0$, $\BDG(\E\dot{\BS}\dot{\BS}^{\top}) = r(\tau)\I_{nd}$, and $\E \dot{\BS}\dot{\BS}^{\top} = \left( 
p - 2((1-\tau)d + \tau) \right)\BZ\BZ^{\top}  + \BSigma_{\BS,\tau}.$

For $\BC$ satisfying $\BC_i\BS_i^{\top} + \BS_i \BC_i^{\top} = 0,$ 
we have
\[
\lag\BL - \BDG(\BL\BS\BS^{\top}), \BS\BC^{\top} \rag = \frac{1}{2}\lag\BL, \BC\BS^{\top} + \BS\BC^{\top} \rag= 0
\]
where $\lag \BDG(\BL\BS\BS^{\top}), \BS\BC^{\top}\rag = 0$
follows from $\BDG(\BS\BC^{\top}) = 0.$ Therefore, it holds that
\begin{equation}\label{eq:LX_perp}
\left\lag \BL, \left(\frac{\alpha}{\lag \BP_{\perp}, \BS\BS^{\top}\rag} - \beta r(\tau)\right) \BS\BS^{\top} - \frac{\delta}{2}(\BC\BS^{\top}+\BS\BC^{\top}) + \beta \BSigma_{\BS,\tau}  \right\rag \geq \frac{\alpha\lag \BL, \BS\BS^{\top}\rag}{\lag \BP_{\perp}, \BS\BS^{\top}\rag}
\end{equation}
for any $\beta > 0$ and $\delta,\alpha\in\RR.$ Suppose $\alpha,\beta,$ and $\delta$ satisfy the constraints in~\eqref{prog:dual2}, and then
\[
\BX_{\perp}: = \left( \left(\frac{\alpha}{\lag \BP_{\perp}, \BS\BS^{\top}\rag} - \beta r(\tau)\right) \BS\BS^{\top} - \frac{\delta}{2}(\BC\BS^{\top}+\BS\BC^{\top}) + \beta \BSigma_{\BS,\tau} \right)_{\perp}\succeq 0.
\]
Therefore, we have
\begin{align*}
& \lag \BL, \BX_{\perp}\rag - \frac{\alpha\lag \BL, \BS\BS^{\top}\rag}{\lag \BP_{\perp}, \BS\BS^{\top}\rag} \\
& \qquad \leq \lambda_{\max}(\BL) \Tr(\BX_{\perp}) - \frac{\alpha \lambda_{d+1}(\BL)}{\lag \BP_{\perp}, \BS\BS^{\top}\rag} \cdot \lag \BP_{\perp}, \BS\BS^{\top}\rag \\
& \qquad \leq \lambda_{\max}(\BL) - \alpha \lambda_{d+1}(\BL)
\end{align*}
where $\Tr(\BX_{\perp})\leq 1$ since $\BX_{\perp}$ satisfies the constraints in~\eqref{prog:dual2}. This implies that
if 
\[
\frac{\lambda_{\max}(\BL)}{\lambda_{d+1}(\BL)} < \alpha, 
\]
then 
\[
\lag \BL, \BX_{\perp}\rag - \frac{\alpha\lag \BL, \BS\BS^{\top}\rag}{\lag \BP_{\perp}, \BS\BS^{\top}\rag} \leq 0
\]
which contradicts~\eqref{eq:LX_perp}, i.e., $\lag \BL, \BX_{\perp}\rag - \alpha\lag \BL, \BS\BS^{\top}\rag/\lag \BP_{\perp}, \BS\BS^{\top}\rag \geq 0.$
Thus the only possible scenario is $\BS\BS^{\top} =  \BZ\BZ^{\top}$, which completes the proof.
\end{proof}

\section{Proof}\label{s:proof}

This section is devoted to the proof of the main theorem and to providing the technical details of the lemmas/propositions in Section~\ref{ss:roadmap}. 
The proof will proceed as follows: in Section~\ref{ss:proof_1}, we first formulate the convex program that minimizes the condition number subject to the coexistence of a spurious second-order critical point and also derive the dual problem, i.e., Proposition~\ref{prop:equivalence1} and~\ref{prop:convexrelax}; Section~\ref{ss:proof_2} gives more technical details on the properties of the linear operator ${\cal L}_{\BS_i}(\cdot)$ and ${\cal P}_{\BS_i,\tau}(\cdot)$ in~\eqref{def:LP}; Section~\ref{ss:proof_3} focuses on the proof of Lemma~\ref{lem:certificate} regarding the construction of a class of feasible dual variables, and reduces the original dual problem into a much simpler one; finally, in Section~\ref{ss:proof_4}, we solve the reduced dual problem and provide a lower bound on the condition number and finish the proof. Section~\ref{ss:proof:optimality} gives the proof of Theorem~\ref{thm:optimality}.

\subsection{Formulation of the condition number minimization problem}\label{ss:proof_1}

We first make the constraints in~\eqref{def:highprog} more precise and establish the equivalence to~\eqref{prog:original}. 
We start with the tightness of SDR: the global optimality of $\BZ\BZ^{\top}$ can be guaranteed by the KKT condition.
\begin{proposition}[Proposition 1 in \cite{L23}]
Let $\BLambda$ be an $nd\times nd$ block diagonal matrix $\BLambda = \blkdiag(\BLambda_{11},\cdots,\BLambda_{nn})$. Suppose $\BLambda$ satisfies
\[
(\BLambda - \BA)\BS = 0,~~\BLambda - \BA \succeq 0,
\]
for some $\BS\in \St(p,d)^{\otimes n}$, then $\BS\BS^{\top}$ is the  global minimizer to the SDR. Moreover, if $\BLambda - \BA$ is of rank $(n -1) d$, i.e., the $(d+1)$-th smallest eigenvalue $\lambda_{d+1}(\BLambda - \BA) > 0$, then $\BS\BS^{\top}$ is the unique global minimizer with rank-$d.$
\end{proposition}

\begin{proof}[\bf Proof of Proposition~\ref{prop:equivalence1}]
The proof is a direct computation by using the Riemannian gradient and Hessian, as well as the proposition above.
Using the proposition above, we can conclude that $\BZ\BZ^{\top}$ is the unique global minimizer if
\[
\BL := \widehat{\BLambda} - \BA\succeq 0,~~\BL\BZ = 0,~~~\widehat{\BLambda} = \BDG(\BA\BZ\BZ^{\top}).
\]
This characterizes the first constraint in~\eqref{def:highprog}, and we proceed with the second one, i.e., characterizing the second-order critical point.

To do this, we first introduce some basic calculation of the Riemannian gradient and Hessian that will be used in the computation, see details in~\cite[Lemma 1 and 2]{L23}. For a more general introduction to the optimization on manifold, the readers may refer to~\cite{AMS08,B23}. 
Consider 
\[
f(\BS) : = -\frac{1}{2}\lag \BA, \BS\BS^{\top}\rag = - \frac{1}{2}\sum_{i,j} \lag \BA_{ij}, \BS_i\BS_j^{\top}\rag,~~~\BS_i\in\St(p,d),~1\leq i\leq n,
\]
and its Riemannian gradient is given by
\[
[\nabla_{\BS}f(\BS)]_i = \sum_{j=1}^n \BA_{ij}\BS_j - \frac{1}{2} \sum_{j=1}^n ( \BA_{ij}\BS_j\BS_i^{\top} + \BS_i\BS_j^{\top}\BA_{ji}) \BS_i
\]
where the tangent space at any $\BS_i\in\St(p,d)$ is
\[
\{\dot{\BS}_i\in\RR^{d\times p}: \dot{\BS}_i\BS_i^{\top} + \BS_i\dot{\BS}_i^{\top} = 0\}.
\]

The first-order optimality condition is
\[
 \sum_{j=1}^n \BA_{ij}\BS_j = \frac{1}{2} \sum_{j=1}^n ( \BA_{ij}\BS_j\BS_i^{\top} + \BS_i\BS_j^{\top}\BA_{ji}).
\]
We denote
\[
\BLambda_{ii} =  \frac{1}{2} \sum_{j=1}^n ( \BA_{ij}\BS_j\BS_i^{\top} + \BS_i\BS_j^{\top}\BA_{ji})
\]
and then
\[
\BLambda := \blkdiag(\BLambda_{11},\cdots,\BLambda_{nn}) = \BDG(\BA\BS\BS^{\top}).
\]
The first order condition becomes:
$(\BLambda - \BA)\BS = 0$. In fact, it holds that
\begin{align*}
 \BLambda - \BA & = \BDG(\BA\BS\BS^{\top}) - \BDG(\BA\BZ\BZ^{\top}) + \BDG(\BA\BZ\BZ^{\top}) - \BA  \\
& = \BDG(\BA\BS\BS^{\top}) - \BDG( \BDG(\BA\BZ\BZ^{\top}) \BS\BS^{\top}) + \BL \\
& = \BL - \BDG(\BL\BS\BS^{\top})
\end{align*}
where $\BL = \BDG(\BA\BZ\BZ^{\top}) - \BA$ and $[\BS\BS^{\top}]_{ii} = \I_d,~1\leq i\leq n$.
Therefore, the first order optimality condition can be represented by using the certificate matrix $\BL$, i.e., 
\[
(\BL - \BDG(\BL\BS\BS^{\top}))\BS = 0.
\]
With a bit more computations, we have the second-order condition as follows:
\[
\dot{\BS}:\nabla^2 f(\BS):\dot{\BS} : = \lag \BLambda - \BA, \dot{\BS}\dot{\BS}^{\top}\rag = \lag\BL - \BDG(\BL\BS\BS^{\top}), \dot{\BS}\dot{\BS}^{\top} \rag\geq 0
\]
for any $\dot{\BS}\in\RR^{nd\times p}$ where $\dot{\BS}^{\top} = [\dot{\BS}_1^{\top},\cdots,\dot{\BS}_n^{\top}]\in\RR^{p\times nd}$ and $\dot{\BS}_i$ is on the tangent space of $\St(p,d)$ at $\BS_i$.
Now we summarize the constraints in~\eqref{def:highprog} as follows.
\begin{enumerate}
\item Sufficient condition for the unique global optimality of $\BZ$:
\[
\BL := \BDG(\BA\BZ\BZ^{\top}) - \BA\succeq 0,~~~\BL\BZ = 0,~~~\lambda_{d+1}(\BL) > 0.
\]

\item Necessary condition for local optimality of $\BS$:
\[
(\BL - \BDG(\BL\BS\BS^{\top}))\BS = 0,~~~\lag  \BL- \BDG(\BL\BS\BS^{\top}) , \dot{\BS}\dot{\BS}^{\top}\rag \geq 0,
\]
for any $\dot{\BS}$ that is on the tangent space of $\St(p,d)^{\otimes n}$ at $\BS.$

\end{enumerate}
Therefore, we arrive at the optimization problem~\ref{prog:original}.
\end{proof}

\begin{proof}[\bf Proof of Proposition~\ref{prop:convexrelax}]
Note that the optimization problem~\eqref{prog:original} is homogeneous in $\BL$, i.e., for any feasible $\BL$, then $c\BL$ is still feasible and achieves the exact same objective value for any $c> 0$. 
Thus we can minimize $t$ such that
\[
\BL_{\perp}\succeq \BP_{\perp},~~~\BL_{\perp}\preceq t\BP_{\perp}.
\]
For the last constraint in~\eqref{prog:original}, we define 
\[
K := \{\dot{\BS}\dot{\BS}^{\top}:  \BS_i\dot{\BS}_i^{\top} + \dot{\BS}_i\BS_i^{\top}= 0,~1\leq i\leq n\}
\]
is a subset of $nd\times nd$ symmetric matrices, 
and $K^*$ is the dual cone of the subset $K$ in $\RR^{nd\times nd}$:
\[
K^* := \{ \BG\in\RR^{nd\times nd}:\BG = \BG^{\top}, \lag \BG, \BK\rag\geq 0,~\forall\BK\in K \}.
\]
The dual cone $K^*$ is convex and closed, and apparently $\BL - \BDG(\BL\BS\BS^{\top})\in K^*$ holds. 

Thus a convex relaxation of~\eqref{prog:original} along with the corresponding dual variables is given by~\eqref{def:primal}.

To find a lower bound of the optimal value, it suffices to compute the dual program, and try to find out a feasible dual solution. 
We proceed to derive the dual program. The Lagrangian function is
\begin{align*}
& {\cal L}(t,\BL, \BX,\BQ,\BB,\BT) \\
& : = t - mt- \lag t \I_{nd} - \BL, \BP_{\perp}\BX\BP_{\perp} \rag - \lag \BL - \I_{nd}, \BP_{\perp}\BQ\BP_{\perp}\rag \\
& ~~+ \lag (\BL_{\perp} - \BDG(\BL_{\perp}\BS\BS^{\top})) \BS, \BB \rag - \lag \BL_{\perp} - \BDG(\BL_{\perp}\BS\BS^{\top}), \BT\rag \\
& = (1 -m- \lag\BP_{\perp}, \BX\rag)t + \lag \BL, \BP_{\perp}(\BX - \BQ + \BB\BS^{\top} - \BT)\BP_{\perp}\rag \\
& ~~ + \lag \BDG(\BL_{\perp}\BS\BS^{\top}), \BT - \BB\BS^{\top}\rag+  \lag\BP_{\perp}, \BQ\rag
\end{align*}
where $\BB\in\RR^{nd\times p}$ and
\begin{align*}
 & \lag \BDG(\BL_{\perp}\BS\BS^{\top}), \BT - \BB\BS^{\top}\rag = \lag \BL_{\perp}\BS\BS^{\top}, \BDG(\BT - \BB\BS^{\top})\rag \\
& ~~~~= \frac{1}{2} \lag \BL_{\perp}, \BDG(\BT - \BB\BS^{\top})\BS\BS^{\top} + \BS\BS^{\top} \BDG(\BT - \BB\BS^{\top}) \rag.
\end{align*}
Then
\begin{align*}
\nabla_{\BL}{\cal L} & =  \BP_{\perp}\left(\BX - \BQ + \frac{1}{2}(\BB\BS^{\top} +\BS\BB^{\top}) - \BT\right)\BP_{\perp} \\
&\qquad + \frac{1}{2}\BP_{\perp}\left( \BDG(\BT - \BB\BS^{\top})\BS\BS^{\top} + \BS\BS^{\top} \BDG(\BT - \BB\BS^{\top})\right)\BP_{\perp}, \\
\nabla_{t}{\cal L} & = 1 - m - \lag \BP_{\perp}, \BX\rag.
\end{align*}
Therefore the dual program becomes~\eqref{prog:dual_reduce}.
\end{proof}

\subsection{Supporting lemmas}\label{ss:proof_2}

We provide a few useful facts about the eigenvalues of ${\cal L}_{\BS_i,\tau}$ and ${\cal P}_{\BS_i,\tau}$ that will be used later.
\begin{lemma}[\bf Spectral properties of ${\cal L}_{\BS_i,\tau}$ and ${\cal P}_{\BS_i,\tau}$]
For any $\BS_i\in\St(p,d),$ the linear operator ${\cal L}_{\BS_i,\tau}(\cdot):\RR^{d\times p}\rightarrow \RR^{d\times p}$ is self-adjoint with eigenvalues:
\begin{itemize}
\item 1, with multiplicity $d(d+1)/2$;
\item $1-2\tau$, with multiplicity $d(d-1)/2$;
\item $0$, with multiplicity $d(p-d)$.
\end{itemize}
The eigenvalues of ${\cal P}_{\BS_i,\tau}$ are $0$, $2\tau$, and $1$ with multiplicities $d(d+1)/2$, $d(d-1)/2$ and $d(p-d)$ respectively. Therefore,  ${\cal P}_{\BS_i,\tau}(\cdot)$ is a positive semidefinite operator for $\tau \geq 0$. 

In particular, it holds that
\begin{equation}\label{eq:pp_sdp}
\lag \PP_{\BS_i,\tau}(\BY), \BY \rag \geq \lag \PP_{\BS_i,0}(\BY), \BY \rag \geq 0,~\forall 1\leq i\leq n,~~~\lag \BZ^{\top}\PP_{\BS,\tau}(\BY), \BY \rag \geq \lag \BZ^{\top}\PP_{\BS,0}(\BY), \BY \rag \geq 0
\end{equation}
 for any $\BY.$
\end{lemma}
\begin{proof}
The adjointness of ${\cal L}_{\BS_i,\tau}$ and ${\cal P}_{\BS_i,\tau}$ follows directly from the definition. We proceed with computing the eigenvalues and corresponding eigenvectors of ${\cal L}_{\BS_i}(\tau).$

\begin{enumerate}
\item For $\BY = \BPsi\BS_i$ for $\BPsi\in\RR^{d\times d}$, then
\[
{\cal L}_{\BS_i,\tau}(\BPsi\BS_i) = \BPsi \BS_i + \tau ( \BPsi^{\top} - \BPsi )\BS_i.
\]
If $\BPsi$ is symmetric, then
\[
{\cal L}_{\BS_i,\tau}(\BPsi\BS_i) = \BPsi\BS_i.
\]
Therefore, the eigenvalue equals 1 with multiplicity $d(d+1)/2$.
 
\item If $\BPsi$ is anti-symmetric, then
\[
{\cal L}_{\BS_i,\tau}(\BPsi\BS_i) = (1-2\tau)\BPsi \BS_i
\]
and the eigenvalue is $1-2\tau$ with multiplicity $d(d-1)/2$.

\item For $\BY = \BPsi(\I_p - \BS_i^{\top}\BS_i)$ with $\BPsi\in\RR^{d\times p}$, then it holds ${\cal L}_{\BS_i,\tau} (\BY) = 0$, and the eigenvalue is 0 with multiplicity $d(p-d)$.
\end{enumerate}
The spectra of $\PP_{\BS_i,\tau}(\cdot)$ follows directly, and the fact that $\lag \PP_{\BS_i,\tau}(\BY), \BY \rag \geq \lag \PP_{\BS_i,0}(\BY), \BY \rag,~\forall 1\leq i\leq n$ immediately gives $\lag \BZ^{\top}\PP_{\BS,\tau}(\BY), \BY \rag \geq \lag \BZ^{\top}\PP_{\BS,0}(\BY), \BY \rag$.
\end{proof}

The next lemma is crucial which computes the covariance of ${\cal L}_{\BS,\tau}$ and ${\cal P}_{\BS,\tau}$.
\begin{lemma}[\bf ``Covariance" of ${\cal L}_{\BS,\tau}$ and ${\cal P}_{\BS,\tau}$]
Let
\begin{equation}\label{def:qr}
\begin{aligned}
q(\tau)  := (d-2)\tau^2 - 2(d-1)\tau + d, ~~~
r(\tau) : =  2(d-1)\tau^2 + p - d,
\end{aligned}
\end{equation}
be two quadratic functions that are nonnegative on $[0,1]$ for any $p\geq d+1.$ Define
\[
\BSigma_{\BS,\tau} : = \E_{\BPhi} \left[ {\cal L}_{\BS,\tau}(\BPhi){\cal L}_{\BS,\tau}(\BPhi)^{\top}\right] 
\]
where $\BPhi$ is a $d\times p$ Gaussian random matrix.
The $(i,j)$-block of $\BSigma_{\BS,\tau}$ equals
\begin{align*}
(\BSigma_{\BS,\tau})_{ij} = (1-\tau)^2\|\BS_i\BS_j^{\top}\|_F^2\I_d + (1-\tau)\tau \left( \BS_i\BS_j^{\top}\BS_j\BS_i^{\top} +  \BS_j\BS_i^{\top}\BS_i\BS_j^{\top}\right) +\tau^2 \lag \BS_i,\BS_j\rag\BS_i\BS_j^{\top}
\end{align*}
and its diagonal block is
\[
\BDG(\BSigma_{\BS,\tau}) = \BDG(\E {\cal L}_{\BS,\tau}(\BPhi) {\cal L}_{\BS,\tau}(\BPhi)^{\top}) 
= (q(\tau)+d\tau^2)\I_{nd}. 
\]
Moreover, it holds that
\begin{equation}\label{eq:LLP}
\begin{aligned}
\Tr(\BSigma_{\BS,\tau})&  =n\left( (1-\tau)^2d^2 + 2(1-\tau)\tau d  + \tau^2d^2\right) 
= nd (q(\tau)+d\tau^2), \\
\lag\BSigma_{\BS,\tau},\BZ\BZ^{\top}\rag 
& = \E \| \sum_{i=1}^n {\cal L}_{\BS_i,\tau}(\BPhi) \|_F^2  
= q(\tau) \|\BS\BS^{\top}\|_F^2  + \tau^2 \|\Tr_d(\BS\BS^{\top})\|_F^2,
\end{aligned}
\end{equation}
where $\Tr_d(\BS\BS^{\top})\in\RR^{n\times n}$ is the partial trace of $\BS\BS^{\top}$, i.e., $[\Tr_d(\BS\BS^{\top})]_{ij} = \lag \BS_i,\BS_j\rag$.
Also
\begin{align*}
\E {\cal P}_{\BS,\tau}(\BPhi) {\cal P}_{\BS,\tau}(\BPhi)^{\top} & = (2(d-1)\tau + p - 2d)\BZ\BZ^{\top} + \BSigma_{\BS,\tau}, \\
\BDG(\E {\cal P}_{\BS,\tau}(\BPhi) {\cal P}_{\BS,\tau}(\BPhi)^{\top}) & = ( 2(d-1)\tau^2+p-d ) \I_{nd} = r(\tau) \I_{nd}.
\end{align*}

\end{lemma}
\begin{proof}
Note that
\begin{align*}
 (\BSigma_{\BS,\tau})_{ij}  & = \E {\cal L}_{\BS_i,\tau}(\BPhi){\cal L}_{\BS_j,\tau}(\BPhi)^{\top} \\
&  = \E \left((1-\tau)\BPhi\BS_i^{\top}\BS_i + \tau\BS_i\BPhi^{\top}\BS_i\right)\left((1-\tau)\BPhi\BS_j^{\top}\BS_j + \tau\BS_j\BPhi^{\top}\BS_j\right)^{\top} \\
& = (1-\tau)^2 \E (\BPhi\BS_i^{\top}\BS_i\BS_j^{\top}\BS_j\BPhi^{\top} )+ (1-\tau)\tau \E (\BS_i \BPhi^{\top}\BS_i\BS_j^{\top}\BS_j\BPhi^{\top} ) \\
& \qquad +  (1-\tau)\tau \E(\BPhi\BS_i^{\top}\BS_i\BS_j^{\top}\BPhi\BS_j^{\top}) + \tau^2 \E (\BS_i\BPhi^{\top}\BS_i\BS_j^{\top}\BPhi\BS_j^{\top} ) \\
&= (1-\tau)^2\|\BS_i\BS_j^{\top}\|_F^2\I_d + (1-\tau)\tau \left( \BS_i\BS_j^{\top}\BS_j\BS_i^{\top} + \BS_j\BS_i^{\top}\BS_i\BS_j^{\top}\right) + \tau^2\lag \BS_i,\BS_j\rag\BS_i\BS_j^{\top}.
\end{align*}
For $i=j$, then
\[
 (\BSigma_{\BS,\tau})_{ii} = d (1-\tau)^2 \I_d + 2(1-\tau)\tau\I_d + d\tau^2\I_d = (2 (d-1)\tau^2 - 2(d-1)\tau + d)\I_d = (q(\tau) + d\tau^2)\I_d
\]
which implies 
$
\Tr(\BSigma_{\BS,\tau}) = nd(q(\tau)+d\tau^2)$
and 
\begin{align*}
\lag \BZ\BZ^{\top}, \BSigma_{\BS,\tau}\rag
& = d(1-\tau)^2 \sum_{i,j}\|\BS_i\BS_j^{\top}\|_F^2 + 2 (1-\tau)\tau \sum_{i,j} \|\BS_i\BS_j^{\top}\|_F^2 + \tau^2\sum_{i,j}\lag \BS_i,\BS_j\rag^2 \\
& = q(\tau)\|\BS\BS^{\top}\|_F^2 + \tau^2 \|\Tr_d(\BS\BS^{\top})\|_F^2
\end{align*}
where $\Tr_d(\BS\BS^{\top})\in\RR^{n\times n}$ is the partial trace of $\BS\BS^{\top}$ whose $(i,j)$-entry is $\lag \BS_i,\BS_j\rag$.

For $\E {\cal P}_{\BS,\tau}(\BPhi) {\cal P}_{\BS,\tau}(\BPhi)^{\top}$, it holds that
\begin{align*}
\E {\cal P}_{\BS_i,\tau}(\BPhi) {\cal P}_{\BS_j,\tau}(\BPhi)^{\top} & = \E (\BPhi - {\cal L}_{\BS_i,\tau}(\BPhi))(\BPhi - {\cal L}_{\BS_j,\tau}(\BPhi))^{\top} \\
& = (2(d-1)\tau + p - 2d )\I_d + (\BSigma_{\BS,\tau})_{ij}
\end{align*}
where
\begin{align*}
\E {\cal L}_{\BS_i,\tau}(\BPhi)\BPhi^{\top} & = (1-\tau)\E \BPhi\BS_i^{\top}\BS_i \BPhi^{\top} + \tau\E\BS_i\BPhi^{\top}\BS_i\BPhi^{\top} 
 = ( - (d-1)\tau + d)\I_d.
\end{align*}
In particular, then diagonal block of $\E \E {\cal P}_{\BS,\tau}(\BPhi) {\cal P}_{\BS,\tau}(\BPhi)^{\top}$ is given by
\begin{align*}
\E {\cal P}_{\BS_i,\tau}(\BPhi) {\cal P}_{\BS_i,\tau}(\BPhi)^{\top} & = (q(\tau) + d\tau^2 + 2(d-1)\tau + p - 2d)\I_d 
= r(\tau)\I_d
\end{align*}
and thus $\BDG(\E {\cal P}_{\BS,\tau}(\BPhi) {\cal P}_{\BS,\tau}(\BPhi)^{\top} ) = r(\tau)\I_{nd}.$
\end{proof}

\subsection{Construction of dual variables in~\eqref{prog:dual_reduce}}\label{ss:proof_3}

\begin{proof}[\bf Proof of Lemma~\ref{lem:certificate}]
We start with showing $\BT$ belonging to the convex hull of $K: = \{\dot{\BS}\dot{\BS}^{\top}: \dot{\BS}_i\BS_i^{\top} + \BS_i\dot{\BS}_i^{\top} = 0,~1\leq i\leq n\}.$ 
Denote $\BE_{k\ell}$ as the $d\times p$ matrix with the only nonzero entry located at $(k,\ell)$ and equal to 1, i.e., $\{\BE_{k\ell}\}$ forms the canonical basis in $\RR^{d\times p}.$ 

By letting $\varphi_{k\ell}$ be i.i.d. $\mathcal{N}(0,1)$ random variables with $1\leq k\leq d,1\leq \ell\leq p$, we have
\[
\BPhi = \sum_{k,\ell} \varphi_{k\ell}\BE_{k\ell},~~\varphi_{k\ell}\overset{\rm i.i.d.}{\sim}\mathcal{N}(0,1)
\]
and then
\begin{align*}
\BT & = \beta \E \PP_{\BS,\tau}(\BPhi)\PP_{\BS,\tau}(\BPhi)^{\top} \\
& = \beta \E \sum_{k,\ell} \varphi_{k\ell}\PP_{\BS,\tau}(\BE_{k\ell})  \sum_{k',\ell'} \varphi_{k'\ell'}\PP_{\BS,\tau}(\BE_{k'\ell'})^{\top} \\
& = \beta \sum_{k,\ell} \PP_{\BS,\tau}(\BE_{k\ell}) \PP_{\BS,\tau}(\BE_{k\ell})^{\top} = \beta ( (2(d-1)\tau+p-2d) \BZ\BZ^{\top} + \BSigma_{\BS,\tau} )
\end{align*}
where $\BDG(\BT) = \beta r(\tau) \I_{nd}.$
Each $ \PP_{\BS,\tau}(\BE_{k\ell}) \PP_{\BS,\tau}(\BE_{k\ell})^{\top}$ is an element in $K$ since $\PP_{\BS_i,\tau}(\BE_{k\ell})$ is in the tangent space of $\St(p,d)$ at $\BS_i$, and thus $\BT$ is in $K$.

The linear constraints in~\eqref{prog:dual_reduce} gives
\begin{align*}
\BX_{\perp} & = \BQ_{\perp} - \frac{1}{2}\BP_{\perp} \left(( \BB - \BDG(\BB\BS^{\top})\BS)\BS^{\top} + \BS( \BB - \BDG(\BB\BS^{\top})\BS)^{\top}  \right) \BP_{\perp} \\
&\qquad + \BT_{\perp} - \frac{1}{2}\BP_{\perp} (\BDG(\BT)\BS\BS^{\top}+\BS\BS^{\top}\BDG(\BT)) \\
& = \left(\frac{\alpha}{\lag \BP_{\perp}, \BS\BS^{\top}\rag} - \beta r(\tau)\right)(\BS\BS^{\top})_{\perp} -\frac{\delta}{2} (\BC\BS^{\top} + \BS\BC^{\top})_{\perp} +\beta (\BSigma_{\BS,\tau})_{\perp}
\end{align*}
where
\[
\BQ = \frac{\alpha\BS\BS^{\top}}{\lag \BP_{\perp}, \BS\BS^{\top}\rag},~~ \BB - \BDG(\BB\BS^{\top})\BS = \delta \BC,~~\BDG(\BC\BS^{\top}) = 0,
\]
and each block $\BC_i = \PP_{\BS_i,\tau}(\BY)$ of $\BC$ is on the tangent space of $\St(p,d)$ at $\BS_i.$

Next, we will  show $\BX_{\perp}\succeq 0$ for some choices of $\BY$ in $\BC$ such that we have a dual feasible set of solutions. 
Consider $\BC = {\cal P}_{\BS,\tau}(\BY)\in\RR^{nd\times p}$
where $\|\BY\|_F = 1$. To prove $\BX_{\perp}\succeq 0$, it suffices to have~\eqref{eq:key_psd}
where
\[
(\BC\BC^{\top})_{\perp} = ((\BZ\BY - {\cal L}_{\BS,\tau}(\BY))(\BZ\BY - {\cal L}_{\BS,\tau}(\BY))^{\top})_{\perp} = ({\cal L}_{\BS,\tau}(\BY){\cal L}_{\BS,\tau}(\BY)^{\top})_{\perp}.
\]
Now we first assume~\eqref{eq:key_psd} and show the lemma:
\begin{align*}
\BX_{\perp} & = \left(  \frac{\alpha}{\lag \BP_{\perp}, \BS\BS^{\top}\rag}  - \beta r(\tau) \right) (\BS\BS^{\top})_{\perp} - \frac{\delta}{2}(\BC\BS^{\top} + \BS\BC^{\top})_{\perp}  + \beta \left(\BSigma_{\BS,\tau}\right)_{\perp} \\
& \succeq \left(  \frac{\alpha}{\lag \BP_{\perp}, \BS\BS^{\top}\rag}  - \beta r(\tau)\right) (\BS\BS^{\top})_{\perp} - \frac{\delta}{2}(\BC\BS^{\top} + \BS\BC^{\top})_{\perp}  + \beta \BP_{\perp}\BC\BC^{\top}\BP_{\perp} \\
& = \BP_{\perp} 
\begin{bmatrix}
\BS & \BC
\end{bmatrix}
\begin{bmatrix}
 \left( \dfrac{\alpha}{\lag \BP_{\perp}, \BS\BS^{\top}\rag}  - \beta r(\tau)\right)\I_p & - \dfrac{\delta}{2}\I_p \\
 -\dfrac{\delta}{2}\I_p & \beta\I_p
\end{bmatrix}
\begin{bmatrix}
\BS & \BC
\end{bmatrix}^{\top}\BP_{\perp}
\end{align*}
is positive semidefinite if 
\[
\begin{bmatrix}
 \dfrac{\alpha}{\lag \BP_{\perp}, \BS\BS^{\top}\rag}  -\beta r(\tau) & - \dfrac{\delta}{2} \\
 -\dfrac{\delta}{2} & \beta
\end{bmatrix} \succeq 0.
\]

We proceed to prove~\eqref{eq:key_psd}.
Let $\BY = \sum_k \sigma_k\bu_k\bv_k^{\top}\in\RR^{d\times p}$ be the SVD, and $\|\BY\|_F^2 =\sum_k \sigma_k^2 =1$.
Also we define
\[
\BY_k :=  {\cal L}_{\BS,\tau}(\bu_k\bv_k^{\top}). 
\]
It holds that
\begin{align*}
{\cal L}_{\BS,\tau}(\BY){\cal L}_{\BS,\tau}(\BY)^{\top}
& = \left( \sum_{k=1}^d\sigma_k {\cal L}_{\BS,\tau}(\bu_k\bv_k^{\top})\right) 
\left( \sum_{\ell=1}^d  \sigma_{\ell}{\cal L}_{\BS,\tau}(\bu_{\ell}\bv_{\ell}^{\top}) \right)^{\top} \\
& = \left( \sum_{k=1}^d \sigma_k \BY_k\right)\left( \sum_{\ell=1}^d \sigma_{\ell} \BY_{\ell}\right)^{\top} \\
& = \sum_{1\leq k,\ell\leq d} \sigma_k\sigma_{\ell} \BY_k\BY_{\ell}^{\top} \preceq \frac{1}{2} \sum_{1\leq k,\ell\leq d} (\sigma_{\ell}^2\BY_k\BY_k^{\top} +\sigma_k^2\BY_{\ell}\BY_{\ell}^{\top}  )  \\
& = \sum_{\ell=1}^d \sigma_{\ell}^2 \sum_{k=1}^d \BY_k\BY_k^{\top} = \sum_{k=1}^d \BY_k\BY_k^{\top}.
\end{align*}
It remains to estimate $\BY_k\BY_k^{\top}$:
\begin{align*}
\BY_k\BY_k^{\top} & = {\cal L}_{\BS,\tau}(\bu_k\bv_k^{\top})  {\cal L}_{\BS,\tau}(\bu_k\bv_k^{\top})^{\top} 
\end{align*}
and then
\[
\sum_{k=1}^d \BY_k\BY_k^{\top} \preceq \sum_{1\leq k\leq d,1\leq\ell\leq p}   {\cal L}_{\BS,\tau}(\bu_k\bv_{\ell}^{\top})  {\cal L}_{\BS,\tau}(\bu_k\bv_{\ell}^{\top})^{\top}
\]
where we add a few more vectors $\{\bv_{\ell}\}_{d+1}^p$ so that $\{\bv_{\ell}\}_{\ell=1}^p$ form an orthonormal basis in $\RR^p.$

Note that  for each pair of $(k,\ell)$, we have
\begin{align*}
 & {\cal L}_{\BS_i,\tau}(\bu_k\bv_{\ell}^{\top})  {\cal L}_{\BS_j,\tau}(\bu_k\bv_{\ell}^{\top})^{\top} \\
& \qquad = ( (1-\tau) \bu_k\bv_{\ell}^{\top}\BS_i^{\top} + \tau \BS_i \bv_{\ell}\bu_k^{\top} )\BS_i \BS_j^{\top}(  (1-\tau) \BS_j \bv_{\ell}\bu_k^{\top} + \tau \bu_k\bv_{\ell}^{\top}\BS_j^{\top} ) \\
& \qquad = (1-\tau)\tau \left( \bu_k\bu_k^{\top}\BS_j\BS_i^{\top}\BS_i\bv_{\ell}\bv_{\ell}^{\top}\BS_j^{\top} 
+ \BS_i\bv_{\ell}\bv_{\ell}^{\top}\BS_j^{\top}\BS_j\BS_i^{\top}\bu_k\bu_k^{\top}\right) \\
&\qquad\qquad + \tau^2 \lag \BS_i\BS_j^{\top}, \bu_k\bu_k^{\top}\rag \BS_i\bv_{\ell}\bv_{\ell}^{\top}\BS_j^{\top} + (1-\tau)^2 \lag \BS_i^{\top}\BS_i\BS_j^{\top}\BS_j, \bv_{\ell}\bv_{\ell}^{\top}\rag\bu_k\bu_k^{\top}
\end{align*}
then 
\begin{align*}
& \sum_{k=1}^d \BY_k\BY_k^{\top}  \preceq
\sum_{1\leq k\leq d,1\leq\ell\leq p}{\cal L}_{\BS,\tau}(\bu_k\bv_{\ell}^{\top})  {\cal L}_{\BS,\tau}(\bu_k\bv_{\ell}^{\top})^{\top} \\
& = \sum_{i,j}\left[ \be_i\be_j^{\top}\otimes \sum_{1\leq k\leq d,1\leq \ell\leq p} {\cal L}_{\BS_i,\tau}(\bu_k\bv_{\ell}^{\top})  {\cal L}_{\BS_j,\tau}(\bu_k\bv_{\ell}^{\top})^{\top}\right]  \\
& = \sum_{i,j} \be_i\be_j^{\top}\otimes  \left[ (1-\tau)^2\|\BS_i\BS_j^{\top}\|_F^2\I_d + (1-\tau)\tau (\BS_i\BS_j^{\top}\BS_j\BS_i^{\top} + \BS_j\BS_i^{\top}\BS_i\BS_j^{\top}) + \tau^2\lag \BS_i,\BS_j\rag\BS_i\BS_j^{\top} \right]  \\
& = \E {\cal L}_{\BS,\tau}(\BPhi){\cal L}_{\BS,\tau}(\BPhi)^{\top}
\end{align*}
provided that  $\sum_{k=1}^d \bu_k\bu_k^{\top}= \I_d$ and $\sum_{\ell=1}^p \bv_{\ell}\bv_{\ell}^{\top} = \I_p.$
\end{proof}

\subsection{Reduction of the dual program $\alpha$ and its solution}\label{ss:proof_4}

\begin{proof}[\bf Proof of Lemma~\ref{lem:alpha_find}]
The proof focuses on solving the reduced dual program~\eqref{prog:dual2}.

\noindent{\bf Step 1: determining $\BY$.} We start with picking a proper $\BY$ in~\eqref{prog:dual2} by looking into $\Tr(\BX_{\perp}) \leq 1.$ 

\begin{equation}\label{eq:linearcons}
\begin{aligned}
\Tr(\BX_{\perp}) & = \alpha -\beta r(\tau) \lag\BP_{\perp}, \BS\BS^{\top}\rag - \delta \lag \BP_{\perp}\BC, \BP_{\perp}\BS\rag + \beta \Tr\left( \BSigma_{\BS,\tau}\right)_{\perp}   \leq 1
\end{aligned}
\end{equation}
where 
\[
\Tr\left( \BSigma_{\BS,\tau}\right)_{\perp} = q(\tau) (nd - n^{-1}\|\BS\BS^{\top}\|_F^2)  + \tau^2 (nd^2 - n^{-1}\|\Tr_d(\BS\BS^{\top})\|_F^2)
\]
follows from~\eqref{eq:LLP}.

Note that for $\BC = \PP_{\BS,\tau}(\BY)$, we have
\begin{align*}
\lag \BP_{\perp}\BC, \BP_{\perp}\BS\rag 
& = \lag \BC, \BS\rag - \frac{1}{n} \lag \BC, \BZ\BZ^{\top}\BS\rag  \\
& = \sum_{i=1}^n \lag \BY - {\cal L}_{\BS_i,\tau}(\BY), \BS_i\rag - n^{-1}\sum_{i=1}^n \lag \BY - {\cal L}_{\BS_i,\tau}(\BY), \BZ^{\top}\BS\rag \\
& =  -\lag \BY, \BZ^{\top}\BS\rag + n^{-1} \sum_{i=1}^n \lag \BY, {\cal L}_{\BS_i,\tau}(\BZ^{\top}\BS)\rag \\
& =- \left\lag \BY, \BZ^{\top}\BS - n^{-1}\BZ^{\top}{\cal L}_{\BS,\tau}(\BZ^{\top}\BS) \right\rag \\
& = - \lag \BY, n^{-1}\BZ^{\top}\PP_{\BS,\tau}(\BZ^{\top}\BS) \rag
\end{align*}
where $\lag \BY - {\cal L}_{\BS_i,\tau}(\BY), \BS_i\rag = 0$ because $(\BY - {\cal L}_{\BS_i,\tau}(\BY))\BS_i^{\top}$ is anti-symmetric.

To maximize $\alpha$, it is natural to  maximize $\lag \BP_{\perp}\BC, \BP_{\perp}\BS\rag$ over $\BY$:
\[
\sup_{\|\BY\|_F =1}\lag \BP_{\perp}\BC, \BP_{\perp}\BS\rag = n^{-1} \left\| \BZ^{\top}\PP_{\BS,\tau}(\BZ^{\top}\BS)   \right\|_F
\]
where $\BY = - \BZ^{\top}\PP_{\BS,\tau}(\BZ^{\top}\BS)/\|\BZ^{\top}\PP_{\BS,\tau}(\BZ^{\top}\BS)\|_F.$
Then the optimal value to the dual program~\eqref{prog:dual2} is further lower bounded by that to the following program:
\begin{equation}\label{prog:finaldual}
\begin{aligned}
\max~& \alpha \\
\text{s.t.}~
& \left(\frac{\alpha}{\lag \BP_{\perp}, \BS\BS^{\top}\rag} -\beta r(\tau) \right)\beta \geq \frac{\delta^2}{4}, \\
& \frac{\alpha}{\lag \BP_{\perp}, \BS\BS^{\top}\rag} \geq \beta r(\tau), ~~\beta \geq 0, \\
& \alpha - \beta r(\tau) \lag\BP_{\perp}, \BS\BS^{\top}\rag - \delta n^{-1} \left\| \BZ^{\top}\PP_{\BS,\tau}(\BZ^{\top}\BS)   \right\|_F  + \beta\Tr(\BSigma_{\BS,\tau})_{\perp} \leq 1.
\end{aligned}
\end{equation}

\noindent{\bf Step 2: solving~\eqref{prog:finaldual} over $\beta$ and $\gamma$.} 
Note that
\[
\frac{\alpha \beta}{\lag \BP_{\perp}, \BS\BS^{\top}\rag r(\tau)} - \beta^2 \geq \frac{\delta^2}{4 r(\tau)}
\]
which is equivalent to
\begin{align*}
 \frac{\alpha^2}{4\lag \BP_{\perp}, \BS\BS^{\top}\rag^2 r(\tau)^2}  \geq \left( \frac{\alpha}{2\lag \BP_{\perp}, \BS\BS^{\top}\rag r(\tau)}-\beta\right)^2 + \frac{\delta^2}{4 r(\tau)}.
\end{align*}


Let 
\begin{align*}
\frac{\alpha \cos\theta }{2\lag \BP_{\perp}, \BS\BS^{\top}\rag r(\tau)}  = \frac{\alpha}{2\lag \BP_{\perp}, \BS\BS^{\top}\rag r(\tau)}-\beta, ~~~\frac{\alpha \sin\theta }{2\lag \BP_{\perp}, \BS\BS^{\top}\rag r(\tau)}  = \frac{\delta}{2\sqrt{ r(\tau)}},
\end{align*}
which means
\begin{equation}\label{def:betadelta}
 \beta = \frac{\alpha( 1 - \cos\theta )}{2\lag \BP_{\perp}, \BS\BS^{\top}\rag r(\tau)},\qquad  \delta = \frac{\alpha\sin\theta}{\lag \BP_{\perp}, \BS\BS^{\top}\rag\sqrt{ r(\tau)}}.
\end{equation}
Finally, we verify if the linear constraint~\eqref{eq:linearcons} holds:
\begin{align*}
& \beta r(\tau) \lag\BP_{\perp}, \BS\BS^{\top}\rag + \delta \lag \BP_{\perp}\BC, \BP_{\perp}\BS\rag - \beta \Tr(\BSigma_{\BS,\tau})_{\perp}   \\
& = \beta \left( r(\tau) \lag\BP_{\perp}, \BS\BS^{\top}\rag - \Tr(\BSigma_{\BS,\tau})_{\perp}  \right)  + \delta n^{-1} \left\| \BZ^{\top}\PP_{\BS,\tau}(\BZ^{\top}\BS)   \right\|_F \\
& =  \frac{\left(  r(\tau) \lag\BP_{\perp}, \BS\BS^{\top}\rag -\Tr(\BSigma_{\BS,\tau})_{\perp}   \right)\alpha\left( 1 - \cos\theta  \right)  }{2\lag \BP_{\perp}, \BS\BS^{\top}\rag r(\tau)}  +  \frac{n^{-1} \left\| \BZ^{\top}\PP_{\BS,\tau}(\BZ^{\top}\BS)   \right\|_F \alpha\sin\theta}{\lag \BP_{\perp}, \BS\BS^{\top}\rag\sqrt{ r(\tau)}}  \\
& = \left( 1 - \frac{  \Tr(\BSigma_{\BS,\tau})_{\perp}  }{\lag \BP_{\perp}, \BS\BS^{\top}\rag r(\tau)}  \right) \frac{\alpha(1 - \cos\theta)}{2}  + \frac{ n^{-1} \left\| \BZ^{\top}\PP_{\BS,\tau}(\BZ^{\top}\BS)   \right\|_F \alpha\sin\theta}{\lag \BP_{\perp}, \BS\BS^{\top}\rag\sqrt{ r(\tau)}} \geq \alpha - 1.
\end{align*}
Multiplying both sides by $\alpha^{-1}$, we have
\begin{align*}
\alpha^{-1} & \geq  1 - \frac{1}{2}\left( 1 -  \frac{\Tr(\BSigma_{\BS,\tau})_{\perp}  }{\lag \BP_{\perp}, \BS\BS^{\top}\rag r(\tau)}  \right)  \left( 1 - \cos\theta  \right)  -  \frac{ n^{-1} \left\| \BZ^{\top}\PP_{\BS,\tau}(\BZ^{\top}\BS)   \right\|_F}{\lag \BP_{\perp}, \BS\BS^{\top}\rag\sqrt{ r(\tau)}} \sin\theta .
\end{align*}
To maximize $\alpha$, we minimize the right hand side over $\theta$:
\begin{align*}
\alpha^{-1} - 1 & \geq  -\frac{1}{2}\left( 1 -  \frac{\Tr(\BSigma_{\BS,\tau})_{\perp} }{\lag \BP_{\perp}, \BS\BS^{\top}\rag r(\tau)}  \right) \\
& \qquad - \sqrt{ \frac{1}{4} \left( 1 -  \frac{\Tr(\BSigma_{\BS,\tau})_{\perp}}{\lag \BP_{\perp}, \BS\BS^{\top}\rag r(\tau)}  \right)^2 + \frac{n^{-2} \left\| \BZ^{\top}\PP_{\BS,\tau}(\BZ^{\top}\BS)   \right\|_F^2}{\lag \BP_{\perp}, \BS\BS^{\top}\rag^2 r(\tau)}  }
\end{align*}
and the optimal $\theta$ also determines $\beta$ and $\delta$ via~\eqref{def:betadelta}. In other words, the largest $\alpha$ satisfying the inequality above is a lower bound of~\eqref{prog:finaldual} and thus~\eqref{prog:dual2}. 

The inequality above is not easy to use, and we choose to represent the largest feasible $\alpha$ in another form.
Let $x = 1 - \alpha^{-1} \in [0,1)$ since $\alpha\geq 1$, and then $x$ satisfies following quadratic inequality:
\[
x^2 -  \left( 1 -  \frac{\Tr(\BSigma_{\BS,\tau})_{\perp} }{\lag \BP_{\perp}, \BS\BS^{\top}\rag r(\tau)}  \right) x \leq
 \frac{n^{-2} \left\| \BZ^{\top}\PP_{\BS,\tau}(\BZ^{\top}\BS)   \right\|_F^2}{\lag \BP_{\perp}, \BS\BS^{\top}\rag^2 r(\tau)}.
\]
Note that $x > 0$, and then
\begin{align*}
\alpha^{-1} & = 1-x \geq  \frac{\Tr(\BSigma_{\BS,\tau})_{\perp} }{\lag \BP_{\perp}, \BS\BS^{\top}\rag r(\tau)}  -  \frac{ n^{-2} \left\| \BZ^{\top}\PP_{\BS,\tau}(\BZ^{\top}\BS)   \right\|_F^2}{\lag \BP_{\perp}, \BS\BS^{\top}\rag^2 r(\tau)}\cdot  \frac{1}{x} \\
& =: \frac{g(\BS,x,\tau)}{ r(\tau)} =\frac{g(\BS,1-\alpha^{-1},\tau)}{ r(\tau)}.
\end{align*}
Therefore, the largest $\alpha$ satisfying the inequality above provides a lower bound of the SDP~\eqref{prog:finaldual} for any fixed $\BS$, i.e., if 
\[
\frac{\lambda_{\max}(\BL)}{\lambda_{d+1}(\BL)} < \alpha,
\]
then $\BS$ is not a second-order critical point.
To ensure a global benign landscape, we need to find an $\alpha$ such that the inequality above holds for any $\BS\in\St(p,d)^{\otimes n}$, and thus we need to have~\eqref{eq:prog_mx00}. 

For any upper bound $M(x,\tau)$ of $\sup_{\BS\in\St(p,d)^{\otimes n}} g(\BS,\bx)$, any feasible solution to~\eqref{eq:prog_mx} is feasible to~\eqref{eq:prog_mx0}. Therefore, the optimal value to~\eqref{eq:prog_mx} provides a lower bound to~\eqref{eq:prog_mx00}.
\end{proof}

\begin{proof}[\bf Proof of Theorem~\ref{thm:alphaM}]
Based on the Lemma~\ref{lem:alpha_find}, we first provide a loose bound $M(x,\tau)$ and an estimation of $\alpha.$
A simple bound $M(x,\tau)$ on $g(\BS,x,\tau)$ in~\eqref{def:gsx} is given by
\begin{align*}
g(\BS,x,\tau) 
& \leq \frac{ \Tr(\BSigma_{\BS,\tau})_{\perp} }{\lag \BP_{\perp}, \BS\BS^{\top}\rag} \\
\text{Use~\eqref{eq:LLP}}~~~& = \frac{q(\tau) (nd - n^{-1}\|\BS\BS^{\top}\|_F^2)  + \tau^2 (nd^2 - n^{-1}\|\Tr_d(\BS\BS^{\top})\|_F^2) }{ nd - n^{-1}\|\BZ^{\top}\BS\|_F^2 } \\
&  \leq  \frac{q(\tau)+d\tau^2 }{d}\cdot \frac{nd^2 - n^{-1}\|\Tr_d(\BS\BS^{\top})\|_F^2}{nd - n^{-1}\|\BZ^{\top}\BS\|_F^2} \\
& \leq  \frac{q(\tau)+d\tau^2}{d}\cdot \frac{nd^2 - n^{-1}\|\Tr_d(\BS\BS^{\top})\|_F^2}{nd - \|\Tr_d(\BS\BS^{\top})\|_F} \\
& \leq (q(\tau)+d\tau^2) \cdot \frac{nd+\|\Tr_d(\BS\BS^{\top})\|_F}{nd} \leq 2 (q(\tau)+d\tau^2) 
\end{align*}
where
\begin{align*}
\|\Tr_d(\BS\BS^{\top})\|_F^2 & =\sum_{i,j} \lag \BS_i,\BS_j\rag^2 =\sum_{i,j} \lag \BS_i\BS_j^{\top},\I_d\rag^2  \leq d\sum_{i,j} \|\BS_i\BS_j^{\top}\|_F^2 = d\|\BS\BS^{\top}\|_F^2, \\
\|\BZ^{\top}\BS\|_F^2 & = \lag \BZ\BZ^{\top}, \BS\BS^{\top}\rag =\sum_{i,j}\lag \BS_i,\BS_j\rag \leq n \|\Tr_d(\BS\BS^{\top})\|_F \leq n^2 d,
\end{align*}
follow from Cauchy-Schwarz inequality. Therefore, by letting $M(\tau) : = 2(q(\tau) + d\tau^2)$ in~\eqref{eq:prog_mx}, then we have
\[
\alpha^{-1} \geq \frac{M(\tau)}{r(\tau)} =  \frac{2(q(\tau)+d\tau^2)}{r(\tau)} \Longleftrightarrow
\alpha_M(p,d,\tau) := \frac{1}{2}\cdot\frac{r(\tau)}{q(\tau)+d\tau^2}. 
\] 

Therefore, it suffices to maximize
\[
\alpha_M(p,d) : = \max_{0\leq \tau\leq 1}~~\alpha_M(p,d,\tau) = \frac{1}{2}\cdot\frac{r(\tau)}{q(\tau)+d\tau^2} =  \frac{1}{2}\cdot\frac{\tau^2 + \frac{p- d}{2(d-1)}}{\tau^2 - \tau + \frac{d}{2(d-1)} }
\]
for any given pair of $(p,d).$
Then
\begin{align*}
\alpha'_M(p,d,\tau) & = \frac{2\tau (\tau^2 - \tau + \frac{d}{2(d-1)}) - (2\tau-1) (\tau^2 + \frac{p-d}{2(d-1)})}{ 2(\tau^2 - \tau + \frac{d}{2(d-1)})^2 } = \frac{-\tau^2  + \frac{(2d-p)\tau}{d-1} + \frac{p-d}{2(d-1)} }{2 (\tau^2 - \tau + \frac{d}{2(d-1)})^2 }
\end{align*}
where $\alpha_M'(p,d,\tau)$ is the derivative w.r.t. $\tau.$
It is easy to check that the maximum is attained at the larger root of 
$\alpha_M'(p,d,\tau) = 0$, i.e., 
\[
\tau^2_*  - \frac{(2d-p)\tau_*}{d-1} - \frac{p-d}{2(d-1)}  = 0.
\]
It also holds that
\[
\alpha_M'\left(p,d,\frac{1}{2}\right) > 0,~~~\alpha_M'(p,d,1) \leq 0,
\]
and 
\[
\alpha_M(p,d,\tau_*) = \frac{r(\tau_*)}{2(q(\tau_*) + d\tau_*^2)} = \frac{1}{2}\cdot \frac{r'(\tau_*)}{q'(\tau_*) + 2d\tau_*} = \frac{\tau_*}{2\tau_*-1}.
\]
Therefore, the maximizer $\tau_*$ satisfies
\[
\frac{1}{2}< \tau_*(p,d) = \frac{1}{2}\left( \frac{2d-p}{d-1} + \sqrt{ \frac{(2d-p)^2}{(d-1)^2} + \frac{2(p-d)}{d-1} }\right) \leq 1.
\]

For $p = d+2$, then $\tau_* = 1.$ For $d\geq 2$, we can see that the optimal value $\alpha_M(p,d)$ is increasing w.r.t. $p$ because 
\[
\alpha_M(p,d,\tau) < \alpha_M(p+1,d,\tau),~~\forall 0\leq \tau\leq 1.
\]
Thus 
\[
\alpha_M(p,d) = \max_{0\leq \tau\leq 1}\alpha_M(p,d,\tau)
\] 
is increasing w.r.t. $p\geq d+2.$
\end{proof}

\subsection{A refined estimation of $\alpha$}\label{ss:proof_5}

To obtain a refined estimation of $\alpha$, we first introduce  the projection operator of any $d\times d$ matrix onto $\Od(d)$ as follows:
\begin{equation}
{\cal P}(\BM) \in \argmin_{\BPhi\in \St(p,d)}\| \BPhi - \BM \|_F^2,~~\forall \BM\in\RR^{d\times p}.
\end{equation}
The global minimizer ${\cal P}(\cdot)$ is not unique if $\BM$ is rank-deficient. In most cases, suppose $\BM = \BU\BSigma\BV^{\top}$ is the SVD of $\BM$ where $\BSigma$ is a diagonal matrix, then we simply set ${\cal P}(\BM) = \BU\BV^{\top}$. By definition, we know that
\[
{\cal P}(\BM)\BM^{\top} = \BM^{\top}{\cal P}(\BM) = \BV\BU^{\top} \BU\BSigma\BV^{\top} = \BV\BSigma\BV^{\top}
\succeq 0.
\]

\begin{lemma}[\bf Refined estimation of~\eqref{def:gsx}]\label{lem:thmmain}
Let $\BR^{\top} = {\cal P}(\BZ^{\top}\BS)$ and $\BR^{\top}\in\St(p,d)$.
Define
\[
v := 1 - \frac{\|\BZ^{\top}\BS\|_F^2}{n^2 d},~~~u :=  vt =  1 - \frac{\|\BS\BR\|_F^2}{nd},~~~w := 1 - \frac{1}{nd}\sum_{i=1}^n \frac{\lag \BS_i, \BZ^{\top}\BS\rag^2}{\|\BZ^{\top}\BS\|_F^2}
\]
and they satisfy
\[
0\leq u\leq w\leq v\leq 1.
\]
In particular, if $d=1$, then $u = w.$
For $\tau \in [0,1]$, it holds that
\begin{align*}
nd - n^{-1}\|\BS\BS^{\top}\|_F^2 & \leq  nd u \left(2 - \frac{pu}{p-d}\right), \\
nd^2 - n^{-1}\|\Tr_d(\BS\BS^{\top})\|_F^2 & \leq  n d^2w\left(2 - \frac{pd w}{pd-1}\right), \\
n^{-1} \left\| \BZ^{\top}\PP_{\BS,\tau}(\BZ^{\top}\BS)   \right\|_F & \geq \sigma^2_{\min}(\BZ^{\top}\BS)d u^2.
\end{align*}
Then $g(\BS,x,\tau)$ in~\eqref{def:gsx} satisfies
\begin{equation}\label{def:gsx_1}
g(\BS,x,\tau) \leq  g(u,v,w,x,\tau) : = q(\tau) \cdot \frac{ u}{v}\left( 2 - \frac{pu}{p-d}\right) + d\tau^2\cdot\frac{w}{v} \left(2 - \frac{pdw}{pd-1}\right)  - \frac{(1-dv)_+u^2}{x \cdot dv^2}.
\end{equation}
\end{lemma}
\begin{proof}

We estimate each term in $g(\BS, x, \tau)$, i.e.,~\eqref{def:gsx}:
\[
 g(\BS,x,\tau) := \frac{ q(\tau)(nd - n^{-1}\|\BS\BS^{\top}\|_F^2) + \tau^2(nd^2 -n^{-1}\|\Tr_d(\BS\BS^{\top})\|_F^2)}{\lag \BP_{\perp}, \BS\BS^{\top}\rag} - \frac{ \left\|\BZ^{\top}\BS - n^{-1}\BZ^{\top}{\cal L}_{\BS,\tau}( \BZ^{\top}\BS) \right\|_F ^2}{ \lag \BP_{\perp}, \BS\BS^{\top}\rag^2 x }
\]
where
\begin{align*}
\Tr(\BSigma_{\BS,\tau})_{\perp} & = q(\tau)(nd - n^{-1}\|\BS\BS^{\top}\|_F^2) + \tau^2(nd^2 -n^{-1}\|\Tr_d(\BS\BS^{\top})\|_F^2), \\
n^{-2}\|\BZ^{\top}\PP_{\BS,\tau}(\BZ^{\top}\BS)\|_F^2 & = n^{-2}\left\| \sum_{i=1}^n (\BZ^{\top}\BS - {\cal L}_{\BS_i,\tau}(\BZ^{\top}\BS))\right\|_F^2=  \left\|\BZ^{\top}\BS - n^{-1}\BZ^{\top}{\cal L}_{\BS,\tau}( \BZ^{\top}\BS) \right\|_F^2.
\end{align*}
\noindent{\bf Estimation of $nd-n^{-1}\|\BS\BS^{\top}\|_F^2$.} Let $\BR^{\top} = \PP(\BZ^{\top}\BS)$. Note $\BR\BR^{\top}$ is a projection matrix, and it holds that
\begin{align*}
\|\BS\BS^{\top}\|_F^2 & = \lag \BS\BS^{\top}, \BS\BS^{\top}\rag\\
& = \lag \BS\BS^{\top}, \BS\BR\BR^{\top}\BS^{\top}\rag + \lag \BS\BS^{\top}, \BS(\I_p - \BR\BR^{\top})\BS^{\top}\rag \\
& \geq  \lag \BS\BR\BR^{\top}\BS^{\top}, \BS\BR\BR^{\top}\BS^{\top}\rag + \lag \BS(\I_p - \BR\BR^{\top})\BS^{\top}, \BS(\I_p - \BR\BR^{\top})\BS^{\top}\rag \\
& = \| \BS\BR\BR^{\top}\BS^{\top}  \|_F^2 + \|\BS(\I_p - \BR\BR^{\top})\BS^{\top}\|_F^2 \\
& \geq \| \BS\BR\BR^{\top}\BS^{\top}  \|_F^2 + \frac{ (nd - \|\BS\BR\|_F^2)^2 }{p-d} \\
& \geq \frac{\|\BS\BR\|_F^4}{d}  + \frac{ (nd - \|\BS\BR\|_F^2)^2 }{p-d} 
\end{align*}
where $\|\BZ\|_F = \|\BS\|_F = \sqrt{nd}$ and the fifth/sixth lines follow from~\eqref{def:Xnorm},
\begin{equation}\label{def:Xnorm}
[\Tr(\BX)]^2 \leq \rank(\BX)\cdot \|\BX\|_F^2
\end{equation}
for any symmetric matrix $\BX.$
Moreover, 
\[
\Tr(\BS(\I_p - \BR\BR^{\top})\BS^{\top}) = \|\BS\|_F^2 - \|\BS\BR\|_F^2 = nd - \|\BS\BR\|_F^2,
\]
and  $\rank(\BS\BR)\leq d$ and $\rank(\BS(\I_p-\BR\BR^{\top})\BS^{\top})\leq p-d$ because of $\BZ^{\top}\BS(\I_p - \BR\BR^{\top}) = 0$.

Thus we have
\begin{equation}\label{eq:est2}
\begin{aligned}
nd - n^{-1}\|\BS\BS^{\top}\|_F^2 & \leq nd - \frac{\|\BS\BR\|_F^4}{nd}  - \frac{ (nd - \|\BS\BR\|_F^2)^2 }{n(p-d)} \\
& = \frac{(nd + \|\BS\BR\|_F^2)(nd - \|\BS\BR\|_F^2)}{nd} - \frac{ (nd - \|\BS\BR\|_F^2)^2 }{n(p-d)} \\
& = (nd - \|\BS\BR\|_F^2)\left( 1 + \frac{\|\BS\BR\|_F^2}{nd} - \frac{nd-\|\BS\BR\|_F^2}{n(p-d)}  \right) \\
& = nd u \left( 2 - u - \frac{du}{p-d} \right) = nd u \left( 2 - \frac{pu}{p-d} \right).
\end{aligned}
\end{equation}

\noindent{\bf Estimation of $nd^2-n^{-1}\|\Tr_d(\BS\BS^{\top})\|_F^2$. }Let $\widetilde{\BR} = \BZ^{\top}\BS/\|\BZ^{\top}\BS\|_F\in\RR^{d\times p}$, and $\widetilde{\br} = \VEC(\widetilde{\BR} )\in\RR^{pd}$ and 
\[
\widetilde{\BS} = 
\begin{bmatrix}
\bs_1^{\top} \\
\vdots \\
\bs_n^{\top}
\end{bmatrix}\in\RR^{n\times pd},~~~\bs_i = \VEC(\BS_i).
\]
Then
\begin{align*}
\|\Tr_d(\BS\BS^{\top})\|_F^2 & = \|\widetilde{\BS}\widetilde{\BS}^{\top}\|_F^2  \geq \|\widetilde{\BS}\widetilde{\br}\widetilde{\br}^{\top}\widetilde{\BS}^{\top}\|_F^2 + \|\widetilde{\BS}(\I_{pd} - \widetilde{\br}\widetilde{\br}^{\top})\widetilde{\BS}^{\top}\|_F^2 \\
& \geq \|\widetilde{\BS}\widetilde{\br}\|_F^4 + \frac{[\Tr(\widetilde{\BS}(\I_{pd} - \widetilde{\br}\widetilde{\br}^{\top})\widetilde{\BS}^{\top})]^2}{pd-1} \\
& = \|\widetilde{\BS}\widetilde{\br}\|_F^4 + \frac{(nd - \|\widetilde{\BS}\widetilde{\br}\|_F^2)^2}{pd-1}
\end{align*}
where  $\widetilde{\br}\widetilde{\br}^{\top}$ is a projection matrix and $\widetilde{\BS}(\I_{pd} - \widetilde{\br}\widetilde{\br}^{\top})\widetilde{\BS}^{\top}$ is of rank at most $pd-1$ because $\bone_n^{\top}\widetilde{\BS}(\I_{pd} - \widetilde{\br}\widetilde{\br}^{\top}) = 0$.

Note that
\begin{align*}
& [\widetilde{\BS}\widetilde{\br}]_i = \lag \widetilde{\bs}_i, \widetilde{\br}\rag = \frac{\lag \VEC(\BS_i), \VEC(\BZ^{\top}\BS)\rag}{\|\BZ^{\top}\BS\|_F} = \frac{\lag \BS_i, \BZ^{\top}\BS\rag}{\|\BZ^{\top}\BS\|_F} \\
& \Longrightarrow \|\widetilde{\BS}\widetilde{\br}\|_F^2 = \frac{\sum_{i=1}^n\lag \BS_i, \BZ^{\top}\BS\rag^2}{\|\BZ^{\top}\BS\|_F^2} = nd (1-w)
\end{align*}
and thus
\[
\|\Tr_d(\BS\BS^{\top})\|_F^2  \geq n^2 d^2 (1-w)^2 + \frac{n^2d^2w^2}{pd-1}.
\]
Therefore, we have
\begin{equation}\label{eq:est3}
nd^2 - n^{-1}\|\Tr_d(\BS\BS^{\top})\|_F^2 \leq nd^2\left( 1 - (1-w)^2 - \frac{w^2}{pd-1}\right) = nd^2w \left(2 - \frac{pd w}{pd-1}\right).
\end{equation}
Here we have $u\leq w$ because
\[
\sum_{i=1}^n \lag \BS_i, \BZ^{\top}\BS\rag^2 = \sum_{i=1}^n \lag \BS_i\BR, \BZ^{\top}\BS\BR\rag^2 \leq \|\BS\BR\|_F^2 \|\BZ^{\top}\BS\|_F^2
\]
and thus it leads to $nd (1-w) \leq nd (1-u)$, i.e., $w\geq u.$

\noindent{\bf Estimation of $\left\|\BZ^{\top}\BS - n^{-1}\BZ^{\top}{\cal L}_{\BS,\tau}(\BZ^{\top}\BS)\right\|^2_F$.}

For any $\tau\geq 0$, we have
\begin{align*}
n^{-2}\left\| \BZ^{\top}\PP_{\BS,\tau}(\BZ^{\top}\BS) \right\|_F^2 
& \geq n^{-2}\| \BZ^{\top}\PP_{\BS,0}(\BZ^{\top}\BS) \|_F^2 \\
& = n^{-2}\left\| \sum_{i=1}^n \BZ^{\top}\BS (\I_p - \BS_i^{\top}\BS_i) \right\|_F^2 \\
& = \|\BZ^{\top}\BS(\I_p - n^{-1}\BS^{\top}\BS)\|_F^2 \\
& \geq \|\BZ^{\top}\BS\BR(\I_d - n^{-1}\BR^{\top}\BS^{\top}\BS\BR)\|_F^2 \\
& \geq \sigma_{\min}^2(\BZ^{\top}\BS) \|\I_d - n^{-1}\BR^{\top}\BS^{\top}\BS\BR\|_F^2 \\
& \geq \frac{\sigma^2_{\min}(\BZ^{\top}\BS)}{d} \cdot [\Tr(\I_d - n^{-1}\BR^{\top}\BS^{\top}\BS\BR)]^2 \\
& = \frac{\sigma^2_{\min}(\BZ^{\top}\BS)}{d} \cdot (d - n^{-1}\|\BS\BR\|_F^2)^2
\end{align*}
where the second line follows from~\eqref{eq:pp_sdp} and $\BZ^{\top}\BS\BR\BR^{\top} = \BZ^{\top}\BS.$

By using the definition of $u$ and $n^2 - \sigma^2_{\min}(\BZ^{\top}\BS) \leq n^2d - \|\BZ^{\top}\BS\|_F^2$, it holds that 
\begin{equation}\label{eq:est1}
\begin{aligned}
n^{-2}\left\| \BZ^{\top}\PP_{\BS,\tau}(\BZ^{\top}\BS) \right\|_F^2  
& \geq \sigma^2_{\min}(\BZ^{\top}\BS) du^2 \\
& \geq ( n^2 - (n^2d - \|\BZ^{\top}\BS\|_F^2) )_+du^2 \\
& = n^2 (1 - d v)_+ d u^2.
\end{aligned}
\end{equation}

Using~\eqref{eq:est2}-\eqref{eq:est1}, and plugging them into~\eqref{def:gsx}, we have
\begin{align*}
g(\BS,x,\tau) & \leq q(\tau) \frac{ nd - n^{-1}\|\BS\BS^{\top}\|_F^2}{nd - n^{-1}\|\BZ^{\top}\BS\|_F^2}  + \tau^2 \frac{nd^2 - n^{-1}\|\Tr_d(\BS\BS^{\top})\|_F^2}{nd - n^{-1}\|\BZ^{\top}\BS\|_F^2}  - \frac{n^{-2}\| \BZ^{\top}\PP_{\BS,\tau}(\BZ^{\top}\BS) \|_F^2}{ (nd - n^{-1}\|\BZ^{\top}\BS\|_F^2)^2 x} \\
& =  q(\tau) \cdot \frac{ nd u}{nd v}\left( 2 - \frac{pu}{p-d}\right) + \tau^2 \frac{nd^2}{nd v}\left(2w - \frac{pdw^2}{pd-1}\right)  - \frac{n^2 (1-dv)_+du^2}{x \cdot n^2d^2v^2} \\
& =  q(\tau) \cdot \frac{ u}{v}\left( 2 - \frac{pu}{p-d}\right) + d\tau^2\cdot\frac{w}{v} \left(2 - \frac{pdw}{pd-1}\right)  - \frac{(1-dv)_+u^2}{x \cdot dv^2}. 
\end{align*}
\end{proof}

\begin{proof}[\bf Proof of Theorem~\ref{thm:main}]
Based on Lemma~\ref{lem:thmmain}, we know that
\[
g(\BS,x,\tau) \leq g(u,v,w,x,\tau)
\]
where $g(u,v,w,x,\tau)$ is given in~\eqref{def:gsx_1}. 
Let 
\begin{align*}
G(x,\tau) & : = \sup_{0\leq u\leq w\leq v\leq 1} g(u,v,w,x,\tau) \\
& =\sup_{0\leq u\leq w\leq v\leq 1}\left\{ q(\tau) \cdot\frac{u}{v} \left( 2 - \frac{pu}{p-d}\right) + d\tau^2 \cdot\frac{w}{v}\left(2 -\frac{pdw}{pd-1} \right)  - \frac{1}{x}\cdot \frac{(1-dv)_+u^2}{dv^2}\right\}.
\end{align*}
Then by applying Lemma~\ref{lem:alpha_find}, we need to find the largest $\alpha_G(p,d,\tau)$ such that
\[
r(\tau)\alpha^{-1} \geq G(1-\alpha^{-1},\tau).
\]
Then we obtain the  optimal value of $\alpha_G(p,d)$ via $\max_{0\leq \tau\leq 1}\alpha_G(p,d,\tau)$, which finishes Theorem~\ref{thm:main}.
\end{proof}

\begin{proof}[\bf Proof of Proposition~\ref{prop:gfun}: the exact form of $G(x,\tau)$]
Consider
\begin{align*}
G(x,\tau) & := \sup_{0\leq u\leq w\leq v\leq 1}\left\{ q(\tau) \cdot \frac{u}{v} \left( 2 - \frac{pu}{p-d}\right) + d\tau^2 \cdot\frac{w}{v}\left(2 - \frac{pdw}{pd-1}\right)  - \frac{(1-dv)_+u^2}{xdv^2}\right\}.
\end{align*}
where $0\leq u\leq w\leq v\leq 1.$ 
Define
\[
s := \frac{w}{v},~~~t := \frac{u}{v},~~~0\leq t\leq s\leq 1,
\]
and
\[
G(x,\tau)  := \sup_{0\leq t\leq s\leq 1,0\leq v\leq 1}\left\{ q(\tau) t \left( 2 - \frac{p tv }{p-d}\right) + d\tau^2 s\left(2 - \frac{pdsv}{pd-1}\right)  - \frac{(1-dv)_+ t^2}{xd}\right\}.
\]

For $v\geq 1/d$, we have $(1-dv)_+ = 0$. 
Then
\begin{align}
& \sup_{0\leq t\leq s\leq 1, v\geq 1/d}  \left\{ q(\tau) t \left( 2 - \frac{p tv }{p-d}\right) + d\tau^2 s\left(2 - \frac{pdsv}{pd-1}\right)  - \frac{(1-dv)_+ t^2}{xd}\right\} \nonumber  \\
& = \sup_{0\leq t\leq s\leq 1,v= 1/d}  \left\{ q(\tau) t \left( 2 - \frac{p t }{d(p-d)}\right) + d\tau^2 s\left(2 - \frac{ps}{pd-1}\right) \right\} \nonumber \\
& = \begin{cases}
\dfrac{p-1}{p}, & p\geq 2,~d = 1,\\
q(\tau)\dfrac{d}{d+1} + d\tau^2 \left(2 - \dfrac{p}{pd-1}\right), & p = d+1,~d\geq 2, \\
q(\tau) \left( 2 - \dfrac{p }{d(p-d)}\right) + d\tau^2 \left(2 - \dfrac{p}{pd-1}\right) , & p\geq d+2,~d \geq 2,
\end{cases} \nonumber  \\
& = \begin{cases}
\dfrac{p-1}{p}, & p\geq 2,~d = 1,\\
2(q(\tau)+d\tau^2) - h(\tau)q(\tau), & p \geq d+1,~d\geq 2,
\end{cases} \label{eq:gv1}
\end{align}
where $q(\tau)+d\tau^2 = 2(d-1)\tau^2-2(d-1)\tau+d$ and $h(\cdot)$ is defined in~\eqref{def:hfun}.

For $v\leq 1/d$, then $(1-dv)_+ = 1-dv$ and 
\begin{align*}
& \sup_{0\leq t\leq s\leq 1, v\leq 1/d}\left\{ q(\tau) t \left( 2 - \frac{p tv }{p-d}\right) + d\tau^2 s\left(2 - \frac{pdsv}{pd-1}\right)  - \frac{(1-dv) t^2}{xd}\right\} \\
& = \sup_{0\leq t\leq s\leq 1, v=0}\left\{ 2q(\tau) t  + 2d\tau^2 s  - \frac{t^2}{xd}\right\} \vee \sup_{0\leq t\leq s\leq 1,v=1/d}  \left\{ q(\tau) t \left( 2 - \frac{p t }{d(p-d)}\right) + d\tau^2 s\left(2 - \frac{ps}{pd-1}\right) \right\}.
\end{align*}
Therefore, the equation above is equivalent to $G(x,\tau)$, and the supremum over $0\leq t\leq s\leq 1$ and $v=1/d$ is 
given in~\eqref{eq:gv1}. Here
\begin{equation}\label{eq:gv2}
\sup_{0\leq t\leq s\leq 1, v=0}\left\{ 2q(\tau) t  + 2d\tau^2 s  - \frac{t^2}{xd}\right\} = 
\begin{cases}
2q(\tau) + 2d\tau^2 - \dfrac{1}{xd}, & xdq(\tau) \geq 1, \\
xdq^2(\tau) + 2d\tau^2, & 0\leq xdq(\tau) \leq 1.
\end{cases}
\end{equation}
We compute $G(x,\tau)$ based on the cases discussed in~\eqref{eq:gv1}.

\noindent{\bf Case 1:} For $d=1$, then 
\[
\sup_{0\leq t\leq s\leq 1, v=0}\left\{ 2q(\tau) t  + 2d\tau^2 s  - \frac{t^2}{xd}\right\} = 
x (1-\tau^2)^2 + 2\tau^2, ~~\forall 0\leq x \leq 1,
\]
where $q(\tau) = 1-\tau^2\leq 1$ for $d=1$ and $xq(\tau)\in [0,1]$. Then
\[
G(x,\tau) = \{x (1-\tau^2)^2 + 2\tau^2\}\vee \frac{p-1}{p}.
\]

\noindent{\bf Case 2:} For $p \geq d+1$ and $d\geq 2$, then we compare~\eqref{eq:gv1} with~\eqref{eq:gv2} and determine which one is larger. Note that under $xdq(\tau)\geq 1$, it holds
\begin{align*}
 2q(\tau) + 2d\tau^2 - \dfrac{1}{xd} \geq 2(q(\tau)+d\tau^2) - h(\tau)q(\tau) & \Longleftrightarrow  h(\tau)  \geq \dfrac{1}{xdq(\tau)}
\end{align*}
where $h(\tau)$ is defined in~\eqref{def:hfun}
and
\begin{align*}
& xdq^2(\tau) + 2d\tau^2 \geq 2( q(\tau)  + d\tau^2) - h(\tau)q(\tau),~xdq(\tau) \leq 1  \Longleftrightarrow 1\geq xdq(\tau) \geq 2 - h(\tau).
\end{align*}
Then
\[
G(x,\tau) = 
\begin{cases}
2q(\tau) + 2d\tau^2 - \dfrac{1}{xd}, & xdq(\tau) \geq \max\left\{1,\dfrac{1}{h(\tau)} \right\} \\
xdq^2(\tau) + 2d\tau^2, & 1\geq xdq(\tau)  \geq  2 - h(\tau), \\
2q(\tau) + 2d\tau^2 - h(\tau)q(\tau), & {\rm otherwise}.
\end{cases}
\]
\end{proof}

\begin{proof}[\bf Proof of Theorem~\ref{thm:main2}]
\noindent{\bf Part I: exact computation of $\alpha_G(p,d,\tau)$.}
For $d = 1$, we have provide the proof right after Theorem~\ref{thm:main2} in Section~\ref{s:main}. Then we proceed with $d\geq 2.$
For $d\geq 2$, we use the form of $G(x,\tau)$ in~\eqref{def:Gfun}. We investigate~\eqref{def:Gfun} based on whether $h(\tau)\geq 1$ or not.
For $p \geq d +1$ with $\tau$ satisfying $h(\tau) \geq 1$ (in particular, for $p = d +1$, then $h(\tau) \geq 1$ holds for any $\tau$), 
then
\begin{equation}
G(x,\tau) = 
\begin{dcases}
2(q(\tau) + d\tau^2) - \frac{1}{xd}, & x \geq \frac{1}{dq(\tau)}, \\
xdq^2(\tau)  + 2d\tau^2, & \frac{2-h(\tau)}{dq(\tau)}\leq x \leq \frac{1}{dq(\tau)}, \\
2(q(\tau) + d\tau^2) - q(\tau) h(\tau), & x \leq \frac{2-h(\tau)}{dq(\tau)}.
\end{dcases} 
\end{equation}
For $p \geq d+2$ with $\tau$ satisfying
$h(\tau) \leq 1,$
then
\begin{equation}
G(x,\tau) = 
\begin{dcases}
2(q(\tau) + d\tau^2) - \frac{1}{xd}, & x \geq \frac{1}{dq(\tau)h(\tau)}, \\
2(q(\tau) + d\tau^2) - q(\tau)h(\tau), & x \leq \frac{1}{dq(\tau)h(\tau)}.
\end{dcases}
\end{equation}

It suffices to show $h(\tau) \geq 1$ and the case with $h(\tau) \leq 1$ is the same.
The argument directly follows from finding the largest $\alpha$ satisfying the following condition in each region:
\[
r(\tau) \alpha^{-1} \geq G(1-\alpha^{-1}, \tau) \Longleftrightarrow r(\tau) (1-x)\geq G(x, \tau)
\]
where $G(x,\tau)$ is increasing in $x$ and $x= 1-\alpha^{-1}$. Since $G(1-\alpha^{-1},\tau)$ is a piece-wise and increasing function in $\alpha$, the strategy is to search the largest feasible $\alpha$ in each piece. Based on this, we can see that if $\alpha_c > 0$ is a feasible solution satisfying $r(\tau) \alpha_c^{-1}\geq G(1-\alpha_c^{-1}, \tau)$, then any $\alpha\leq \alpha_c$ is feasible.

\noindent{\bf Case 1:} 
Note that
it holds
\[
G(x,\tau) = g(1,0,x,\tau) = 2(q(\tau)+d\tau^2) - \frac{1}{xd},~~~\forall x \geq \frac{1}{dq(\tau)}.
\]
Then
\begin{align*}
r(\tau) \alpha^{-1} \geq G(1-\alpha^{-1}, \tau) 
& \Longleftrightarrow r(\tau) \alpha^{-1} \geq 2(q(\tau)+d\tau^2) - \frac{1}{(1-\alpha^{-1})d} \\
&  \Longleftrightarrow  r(\tau) (\alpha - 1) \geq 2(q(\tau)+d\tau^2) \alpha (\alpha-1) - d^{-1}\alpha^2 \\
& \Longleftrightarrow (2 (q(\tau)+d\tau^2) - d^{-1}) \alpha^2 - (2(q(\tau)+d\tau^2) + r(\tau))\alpha  + r(\tau) \leq 0.
\end{align*}
Thus $\alpha$ is the larger root of the quadratic equation that satisfies $(1-\alpha^{-1})dq(\tau) > 1.$ If such a root does not satisfy $(1-\alpha^{-1})dq(\tau) > 1$, it means a feasible $\alpha$ does not exist in this region.

\noindent{\bf Case 2:}
For 
\[
G(x, \tau) = xdq^2(\tau)  + 2d\tau^2,~~\forall \frac{2 - h(\tau)}{dq(\tau)} \leq x  \leq  \frac{1}{dq(\tau) },
\]
we have
\begin{align*}
r(\tau)\alpha^{-1}  \geq G(1-\alpha^{-1}, \tau) 
& \Longleftrightarrow r(\tau) \alpha^{-1}\geq dq^2(\tau)(1-\alpha^{-1})  + 2d\tau^2 \\
& \Longleftrightarrow  (dq^2(\tau) + 2d\tau^2)\alpha \leq dq^2(\tau) + r(\tau)  \Longleftrightarrow \alpha \leq  \frac{ dq^2(\tau) + r(\tau) }{ dq^2(\tau) + 2d\tau^2}.
\end{align*}

\noindent{\bf Case 3:} 
For
\[
G(x,\tau) = 2(q(\tau) + d\tau^2) - q(\tau) h(\tau),~~\forall x \leq \frac{2-h(\tau)}{dq(\tau)},
\]
we have
\begin{align*}
r(\tau) \alpha^{-1} \geq M(1-\alpha^{-1},\tau) & \Longleftrightarrow r(\tau) \alpha^{-1} \geq  2 (q(\tau)+d\tau^2) - q(\tau)h(\tau) \\
& \Longleftrightarrow \alpha \leq \frac{r(\tau)}{2 (q(\tau)+d\tau^2) - q(\tau)h(\tau) }.
\end{align*}
Combining them gives the expression of $\alpha_G(p,d).$

\noindent{\bf Part II: Proof of $\alpha_G(p,d,\tau) > 1$ for $p\geq d+2$ for $d\geq 2$.}
Next, we will show that
\[
\alpha_G(p,d): = \sup_{0\leq \tau\leq 1} \{\alpha:  r(\tau)\alpha^{-1} \geq G(1-\alpha^{-1},\tau) \} > 1
\]
for any $p\geq d+2$. 
First, note that for $\tau_{\eps} = 1-\eps$ with sufficiently small $\eps>0$, then
\[
\frac{1}{dq(\tau_{\eps})}  > 1,~~~\frac{2-h(\tau_{\eps})}{dq(\tau_{\eps})} < 0
\] 
which corresponds to~\eqref{eq:alphaG2:2}.
With this choice of $\tau_{\eps}$, $G(x,\tau_{\eps}) = xdq^2(\tau_{\eps}) + 2d\tau_{\eps}^2$ and 
\begin{align*}
r(\tau_{\eps})\tau^{-1}\geq G(1-\alpha^{-1},\tau_{\eps}) &\Longleftrightarrow (1-\alpha^{-1})dq^2(\tau_{\eps}) + 2d\tau_{\eps}^2 \leq r(\tau_{\eps}) \alpha^{-1} \\
& \Longleftrightarrow \alpha_G(p,d,\tau_{\eps}) = \frac{r(\tau_{\eps}) + dq^2(\tau_{\eps})}{2d\tau_{\eps}^2+ dq^2(\tau_{\eps})}.
\end{align*}
Therefore, to make $\alpha_G(p,d,\tau_{\eps})>1$, it suffices to 
have $r(\tau_{\eps}) > 2d\tau_{\eps}^2$, i.e.,
\[
r(\tau_{\eps}) > 2d\tau_{\eps}^2 \Longleftrightarrow p > d+ 2\tau_{\eps}^2.
\]
Note that $\tau_{\eps}^2<1$, and then $p \geq d+2$ guarantees $\alpha_G(p,d) >\alpha_G(p,d,\tau_{\eps}) > 1$.

\vskip0.25cm

\noindent{\bf Part III: Proof of $\alpha_G(p,d,\tau) > 1$ for $p=d+1$ for $d=2$ or $3$.}

For $p = d+1$, then $h(\tau) > 1$, and it is easier to check the feasibility of $x>0$.  Once we have a feasible $x$, we can compute $\alpha$ via $\alpha= 1- x^{-1}.$
 To check the feasibility of $0 < x< 1$ that satisfies $r(\tau)(1-x)\geq G(x,\tau)$, it suffices to only look into the third case of~\eqref{def:Gfun1}, i.e., 
 \[
 G(x,\tau) = 2(q(\tau)+d\tau^2) - q(\tau)h(\tau),~~~x \leq \frac{2-h(\tau)}{dq(\tau)},~~h(\tau) = 2 - \frac{d}{d+1} - \frac{\tau^2}{q(\tau)}\left(2 -\frac{pd}{pd-1}\right). 
 \]

 The existence of $x\in [0,1]$ satisfying the inequalities above means
\begin{align*}
r(\tau) \geq G(x,\tau)&\Longleftrightarrow  2(d-1)\tau^2+1 \geq  \frac{d}{d+1}\cdot q(\tau)   + d\tau^2 \left(2 - \frac{p}{pd-1} \right), \\
2-h(\tau) \geq 0 & \Longleftrightarrow  \dfrac{pd}{pd-1}\cdot \dfrac{\tau^2}{q(\tau)} \leq \frac{d}{d+1},
\end{align*}
where $r(\tau) = 2(d-1)\tau^2 + 1.$
Then a feasible $\alpha_G(p,d,\tau)\geq 1$ exists for some $\tau$ if there exists $\tau\in [0,1]$ such that
\begin{align*}
\frac{pd}{pd-1}\tau^2 \leq q(\tau) \cdot \frac{d}{d+1}    \leq 1 -  \frac{pd-2}{pd-1}\tau^2
\end{align*}
where $q(\tau) = (d-2)\tau^2 - 2(d-1)\tau + d$ and $p=d+1.$
It is straightforward to show that the inequality above has a feasible solution if $d = 2$ and $d=3$. 
\end{proof}

\begin{proof}[\bf Proof of Theorem~\ref{thm:main2-sim2}]
For $d\geq 2$, we know from~\eqref{def:Gfun} that
\[
G(x,\tau) = 2q(\tau) + 2d\tau^2 - \frac{1}{dx}  \leq 2(q(\tau)+d\tau^2) - \frac{2-x}{d} = : \widetilde{M}(x,\tau),~~\forall xdq(\tau)\geq \max\{1, 1/h(\tau)\}
\]
where $x + x^{-1}\geq 2$.
Then we solve for $\alpha$ based on this simpler upper bound $\widetilde{M}(x,\tau)$ by
\[
r(\tau) \alpha^{-1} \geq \widetilde{M}(x,\tau) =  2(q(\tau)+d\tau^2) - \frac{2-x}{d} = 2(q(\tau)+d\tau^2) - d^{-1} - d^{-1}\alpha^{-1}
\]
and
\[
\alpha_{\widetilde{M}}(p,d,\tau) = \frac{r(\tau) + d^{-1}}{2(q(\tau)+d\tau^2) - d^{-1}}. 
\] 
Next, we will show that $\max_{0\leq \tau\leq 1}\alpha_{\widetilde{M}}(p,d,\tau)$ satisfies the constraint in of~\eqref{eq:alphaG2:1} or~\eqref{eq:alphaG3:1}.
We first find the maximizer of the estimation $\alpha_{\widetilde{M}}(p,d,\tau)$ which provides the best lower bound:
\begin{align*}
\alpha_{\widetilde{M}}(p,d,\tau) & :=  \frac{1}{2}\cdot \frac{2(d-1)\tau^2 + p -d + d^{-1}}{ 2(d-1)\tau^2 - 2(d-1)\tau + d - d^{-1}/2 } = \frac{1}{2}\cdot  \frac{\tau^2 + \frac{p-d+d^{-1}}{2(d-1)}}{ \tau^2 - \tau + \frac{2d-d^{-1}}{4(d-1)} }.
\end{align*}
In fact, similar to Theorem~\ref{thm:alphaM}, one can check that $\alpha_{\widetilde{M}}(p,d,\tau)$ has one maximizer in the interval $(0,1)$.
Then we aim to find the optimal value.
The maximizer satisfies:
\[
\frac{\tau^2 + \frac{p-d+d^{-1}}{2(d-1)}}{ \tau^2 - \tau + \frac{2d-d^{-1}}{4(d-1)} } = \frac{2\tau}{2\tau - 1}
\]
which equals
\begin{align*}
\tau^2 - \frac{4d-2p-3d^{-1}}{2(d-1)} \tau- \frac{p-d+d^{-1}}{2(d-1)} = 0.
\end{align*}
Let $\tau_*$ be the positive root of the quadratic equation. Then 
\[
\alpha_{\widetilde{M}}(p,d) : = \max_{0\leq \tau\leq 1}\alpha_{\widetilde{M}}(p,d,\tau)= \frac{\tau_*}{2\tau_*-1}
\]
is a lower bound of $\sup_{0\leq \tau\leq 1}\alpha_G(p,d,\tau)$ provided that the constraint in~\eqref{eq:alphaG2:1} or~\eqref{eq:alphaG3:1} holds, i.e.,
\begin{equation}\label{eq:condeasy}
dq(\tau_*)(1 - \alpha_{\widetilde{M}}(p,d)^{-1}) = dq(\tau_*)\cdot \frac{1-\tau_*}{\tau_*}\geq \max\{  h(\tau_*),1 \}.
\end{equation}

Let 
\[
\tau_c = \frac{\sqrt{d}}{\sqrt{d} + 1}.
\]
We will show that 
\begin{equation}\label{eq:condeasy1}
\tau_* \leq \frac{\sqrt{d}}{\sqrt{d} + 1} =:\tau_c,~~~~dq(\tau_*)\cdot \frac{1-\tau_*}{\tau_*}\geq \max\{  h(\tau_*),1 \}
\end{equation}
holds for
\begin{equation}\label{eq:condp}
p\geq \left(1 + \frac{2\sqrt{d}}{d-1}\right)d + 2.
\end{equation}

We start with proving the first inequality in~\eqref{eq:condeasy1}.
The root $\tau_*$ satisfies $\tau_*\leq \tau$ for some $\tau \in (1/2,1)$ if
\[
\tau^2 - \frac{(4d-2p-3d^{-1})\tau}{2(d-1)} - \frac{p-d+d^{-1}}{2(d-1)}  \geq 0
\]
which is
\begin{equation}\label{eq:pest}
p  \geq \frac{-2(d-1)\tau^2 + (4d-3d^{-1})\tau -d+d^{-1}}{2\tau-1} 
 =  \left( 1 + \frac{2\tau(1-\tau)}{2\tau-1}\right)d + \frac{2\tau^2 - 3d^{-1}\tau + d^{-1}}{2\tau-1}.
\end{equation}

For this choice of $\tau = \tau_c$, plugging it into~\eqref{eq:pest} gives
\begin{align*}
& \frac{2\tau(1-\tau)}{2\tau-1} = \frac{2\sqrt{d} \cdot (\sqrt{d} + 1)^{-1} }{ \sqrt{d} - 1 } = \frac{2\sqrt{d}}{d-1}, \\
& \frac{2\tau^2 - 3d^{-1}\tau + d^{-1}}{2\tau-1} = \frac{\frac{2d}{(\sqrt{d}+1)^2} - \frac{3\sqrt{d}}{d (\sqrt{d}+1)} + \frac{1}{d}}{ \frac{2\sqrt{d}}{\sqrt{d}+1}  - 1} =2 -  \frac{ \sqrt{d} - 1}{  d(d-1)}  < 2, 
\end{align*}
and then $\tau_*\in [1/2, \sqrt{d}/(\sqrt{d}+1)]$ is guaranteed by~\eqref{eq:condp}.

For the second inequality in~\eqref{eq:condeasy1}, it suffices to check it at $\tau = \tau_c$ because $q(\tau)(1-\tau)/\tau$ and $h(\tau)$ are decreasing and increasing in $\tau$ respectively. 
Note that
\[
h(\tau)  = \frac{p}{d(p-d)} + \frac{pd}{pd-1}\cdot\frac{\tau^2}{q(\tau)} \leq \frac{1}{d} + \frac{\tau^2}{q(\tau)} + \frac{1}{p-d}\left(1 + \frac{\tau^2}{dq(\tau)}\right)
\]
where
\[
\frac{1}{pd-1} < \frac{1}{p-d}\frac{1}{d}.
\]
Therefore, it suffices to prove:
\begin{align*}
dq(\tau)\cdot \frac{1-\tau}{\tau} > 1,~~~dq(\tau)\cdot \frac{1-\tau}{\tau}  \geq \frac{1}{d} + \frac{\tau^2}{q(\tau)} + \frac{1}{p-d}\left(1 + \frac{\tau^2}{dq(\tau)}\right)
\end{align*}
so that the second inequality in~\eqref{eq:condeasy1} holds.

This is equivalent to the following three inequalities
\begin{equation}\label{eq:four}
\begin{aligned}
& dq(\tau) > \frac{\tau}{1-\tau}, \\
& \frac{dq(\tau)(1-\tau)}{\tau} -  \frac{1}{d} - \frac{\tau^2}{q(\tau)}  > 0,\\
& p > \left(\frac{dq(\tau)(1-\tau)}{\tau} -  \frac{1}{d} - \frac{\tau^2}{q(\tau)} \right)^{-1} \left(1+\frac{1}{d} \cdot \frac{\tau^2}{q(\tau)}\right) + d, 
\end{aligned}
\end{equation}
at $\tau = \tau_c.$

For $d=2$, numerically it is verified that the three inequalities hold. We prove they hold for $d\geq 3.$
For $\tau = \sqrt{d}/(\sqrt{d}+1)$, then
\[
\frac{\tau^2}{q(\tau)} = \frac{\sqrt{d}}{2+\sqrt{d}},~~~\tau(1-\tau) = \frac{\sqrt{d}}{(\sqrt{d}+1)^2}. 
\]
We check~\eqref{eq:four} one by one:
\begin{align*}
dq(\tau) > \frac{\tau}{1-\tau} & \Longleftrightarrow \frac{dq(\tau)}{\tau^2}\cdot \tau(1-\tau) > 1 \Longleftrightarrow
d\cdot \frac{2+\sqrt{d}}{\sqrt{d}}\cdot\frac{\sqrt{d}}{(\sqrt{d}+1)^2} > 1, \\
 \frac{dq(\tau)(1-\tau)}{\tau} -  \frac{1}{d} - \frac{\tau^2}{q(\tau)}   & = \frac{d(2+\sqrt{d})}{(\sqrt{d}+1)^2} - \frac{1}{d} - \frac{\sqrt{d}}{2+\sqrt{d}}  \geq \frac{d^2-2}{d(2+\sqrt{d})} > 0,
\end{align*}
for any $d\geq 3$. For the third inequality in~\eqref{eq:four}, using the inequality above gives
\[
p >\frac{d(2+\sqrt{d})}{d^2-2} \left(1 + \frac{\sqrt{d}}{d(2+\sqrt{d})} \right) + d = \frac{d(2+\sqrt{d}) + \sqrt{d}}{d^2-2} + d.
\]
It remains to show the right hand side above is smaller than that of~\eqref{eq:condp}, i.e., the inequality above is implied by~\eqref{eq:condp}:
\begin{align*}
& \left(1 + \frac{2\sqrt{d}}{d-1}\right)d+2 - \left(\frac{d(2+\sqrt{d}) + \sqrt{d}}{d^2-2} + d\right) \\
& \qquad=  \frac{2\sqrt{d}d}{d-1} +2 - \frac{d(2+\sqrt{d}) + \sqrt{d}}{d^2-2}   > \frac{2\sqrt{d}d(d+1) - 2d(d+1) }{d^2-2} + 2 > 0
\end{align*}
where $d(2+\sqrt{d}) + \sqrt{d} < 2d (d+1).$ 
\end{proof}

\subsection{Proof of Theorem~\ref{thm:optimality}}\label{ss:proof:optimality}

\begin{proof}[\bf Proof of Theorem~\ref{thm:optimality}(a)]
The proof follows directly from Lemma~\ref{lem:thmmain} and~\ref{lem:alpha_find}. 
Suppose $\BZ^{\top}\BS = 0$, and then $v = 1 - \|\BZ^{\top}\BS\|_F^2/(n^2d) = 1$ and~\eqref{def:gsx_1} gives
\begin{align*}
\sup_{\BS:\BZ^{\top}\BS = 0} g(\BS,x,\tau) & \leq  \sup_{0\leq u\leq w\leq v\leq 1, v=1} g(u,v,w,x,\tau) \\
& =   \sup_{0\leq u\leq w\leq 1, v=1} q(\tau) u\left( 2 - \frac{pu}{p-d}\right) + d\tau^2 w \left(2 - \frac{pdw}{pd-1}\right)    \\
& =  \frac{p-d}{p} \cdot q(\tau) + \frac{pd-1}{pd}\cdot d\tau^2
\end{align*}
where $(1-dv)_+ = 0$ for $v=1.$
By using~\eqref{lem:alpha_find}, we have
\[
r(\tau) \alpha^{-1} \geq  \frac{p-d}{p} \cdot q(\tau) +\frac{pd-1}{pd}\cdot d\tau^2
\]
and we set
\[
\alpha_{\BZ^{\top}\BS=0}(p,d,\tau) = \frac{ r(\tau) }{ \dfrac{p-d}{p}\cdot q(\tau) + d\tau^2\left(1 - \frac{1}{pd}\right) }.
\]
Then we optimize $\alpha(p,d,\tau)$ over $\tau\in [0,1]$ for any fixed pair of $(p,d)$:
\begin{align*}
\alpha_{\BZ^{\top}\BS= 0}(p,d) & = \sup_{0\leq \tau\leq 1}\alpha(p,d,\tau) \\
& =\sup_{0\leq \tau\leq 1} \frac{2(d-1)\tau^2 + p - d}{ \dfrac{p-d}{p} ((d-2)\tau^2 - 2(d-1)\tau + d) + \frac{pd-1}{pd}d\tau^2 } \\
& =  \frac{\frac{d-1}{2} + p - d}{ \frac{p-d}{p} (\frac{d-2}{4} - (d-1) + d) + \frac{pd-1}{pd} \frac{d}{4}}   \\
& = \frac{ 4p( p - \frac{d+1}{2})}{ (p-d)(d+2)  +  pd -1}   = \frac{2p}{d+1}
\end{align*}
where the supremum is attained at $\tau = 1/2$.

\noindent{\bf Proof of Theorem~\ref{thm:optimality}(b)}
Consider $\BS\in\St(p,d)^{\otimes n}$ of the following explicit form:
\[
\{\BS_i\}_{i=1}^n= \{ \diag(\pm 1,\cdots,\pm 1) [\I_d, \bzero_{d\times (p-d)}] \BM^k: 1\leq k\leq p \} 
\]
where $\BM$ is a fixed permutation matrix of size $p\times p$ and $n = 2^dp.$
It holds that
\[
\BS^{\top} \BS= \frac{nd}{p} \I_p,~~~\BZ^{\top}\BS = 0
\]
and moreover if $\BQ\in \{\BS_i\}_{i=1}^n$, then so is $-\BQ$.

Then we let 
\[
\BL = \I_{nd} - \frac{\BZ\BZ^{\top}}{n} + t \left( \I_{nd} - \frac{\BZ\BZ^{\top}}{n} - \frac{p\BS\BS^{\top}}{nd}\right).
\]
We will show that if
\[
t \geq \frac{2p}{d+1}-1,
\]
then $\BS$ is a second-order critical point and $\BZ$ is the global minimizer. In this case, the largest and $(d+1)$-th smallest eigenvalue of $\BL$ are $1 + t$ and 1, and the condition number satisfies
\[
\alpha = \frac{\lambda_{\max}(\BL)}{\lambda_{d+1}(\BL)} = \frac{2p}{d+1}.
\]

By construction, we know that $\BL\BZ = 0$ and $\BL \BS = \BS$ Since $\BS\in\RR^{nd\times p}$ consists of $p$ eigenvectors of $\BL$, then $\BL$ is positive semidefinite for any $t\geq 0.$ This implies that $\BZ$ is the unique global minimizer.

To show that $\BS$ is a second-order critical point, it remains to argue
\[
\lag \BL - \BDG(\BL\BS\BS^{\top}), \dot{\BS}\dot{\BS}^{\top}\rag \geq 0
\]
where $\BDG(\BL\BS\BS^{\top}) = \I_{nd}$ and $(\BL - \BDG(\BL\BS\BS^{\top}))\BS = 0$. This second-order critical point is equivalent to
\begin{align*}
&  \left\lag - \frac{\BZ\BZ^{\top}}{n} + t \left( \I_{nd} - \frac{\BZ\BZ^{\top}}{n} - \frac{p\BS\BS^{\top}}{nd}\right), \dot{\BS}\dot{\BS}^{\top}\right\rag\geq 0  \\
& \Longleftrightarrow \|\dot{\BS}\|_F^2  \geq \frac{1+t^{-1}}{n}\|\BZ^{\top}\dot{\BS}\|_F^2 + \frac{p}{nd} \|\BS^{\top}\dot{\BS}\|_F^2 
\end{align*}

For any tangent vector $\dot{\BS}_i$ of $\St(p,d)$ at $\BS_i$, it can be written as 
\[
\VEC(\dot{\BS}_i) = \VEC\left(\BY_i - \frac{1}{2}(\BS_i\BY_i^{\top}+\BY_i\BS_i^{\top})\BS_i\right) = \BPi_i\VEC(\BY_i)
\]
for some matrix $\BY_i\in\RR^{d\times p}$ where $\BPi_i\in\RR^{pd\times pd}$ is a projection matrix. Note that
\begin{align*}
\|\BZ^{\top}\dot{\BS}\|_F^2 & = \left\|\sum_{i=1}^n \VEC(\dot{\BS}_i)\right\|_F^2 = \left\|\sum_{i=1}^n \BPi_i \VEC(\BY_i)\right\|_F^2 =\| (\BZ^{\top}\otimes \I_p)\BPi \by\|_F^2, \\
\|\BS^{\top}\dot{\BS}\|_F^2 & = \left\|\sum_{i=1}^n \BS_i^{\top}\dot{\BS}_i\right\|_F^2 = \left\| \sum_{i=1}^n (\I_p\otimes \BS_i^{\top})\VEC(\dot{\BS}_i)  \right\|_F^2 \\
&  = \left\| \sum_{i=1}^n (\I_p\otimes \BS_i^{\top})\BPi_i\VEC(\BY_i)  \right\|^2 = \| [\I_p\otimes \BS_1^{\top},\cdots,\I_p\otimes \BS_n^{\top}]\BPi\by \|_F^2
\end{align*}
where 
\[
\BPi = \blkdiag(\BPi_1,\cdots,\BPi_n)\in\RR^{npd\times npd},~~\by = \begin{bmatrix}
\VEC(\BY_1) \\
\vdots\\
\VEC(\BY_n)
\end{bmatrix}\in\RR^{npd}.
\]
Therefore, it is equivalent to show:
\begin{equation}\label{eq:opt_ss}
\BPi\succeq \frac{1+t^{-1}}{n}\BPi (\BZ\BZ^{\top}\otimes \I_p)\BPi + \frac{p}{nd}\BPi [\I_p\otimes \BS_1^{\top},\cdots,\I_p\otimes \BS_n^{\top}]^{\top}[\I_p\otimes \BS_1^{\top},\cdots,\I_p\otimes \BS_n^{\top}]\BPi.
\end{equation}

First we show that $[\I_p\otimes \BS_1^{\top},\cdots,\I_p\otimes \BS_n^{\top}]\BPi (\BZ\otimes \I_p)  = 0$. For any $\bx\in\RR^{pd}$ and $\VEC(\BX) = \bx$, it holds that 
\begin{align*}
&  [\I_p\otimes \BS_1^{\top},\cdots,\I_p\otimes \BS_n^{\top}]\BPi (\BZ\otimes \I_p) \bx = \sum_{i=1}^n (\I_p\otimes \BS_i^{\top}) \BPi_i \bx \\
 & = \sum_{i=1}^n (\I_p\otimes \BS_i^{\top})\VEC\left(  \BX - \frac{1}{2}(\BS_i\BX^{\top}+\BX\BS_i^{\top})\BS_i \right) \\
 & = \sum_{i=1}^n \VEC\left( \BS_i^{\top}\BX - \frac{1}{2} \BS_i^{\top}(\BS_i\BX^{\top}+\BX\BS_i^{\top})\BS_i\right) \\
 & = -\sum_{i=1}^n \VEC\left( \BS_i^{\top}\BX - \frac{1}{2} \BS_i^{\top}(\BS_i\BX^{\top}+\BX\BS_i^{\top})\BS_i\right) 
\end{align*}
for any $\bx\in\RR^{pd}$ because both $\BS_k$ and $-\BS_k$ are in the set $\{\BS_i\}_{i=1}^n.$ 

Also, we have
\[
\|[\I_p\otimes \BS_1^{\top},\cdots,\I_p\otimes \BS_n^{\top}]\BPi\|^2 \leq \left\| \sum_{i=1}^n \I_p\otimes \BS_i^{\top}\BS_i \right\|=  \| \BS^{\top}\BS \| = \frac{nd}{p}.
\]
Next, we compute 
\[
(\BPi(\BZ\otimes \I_p))^{\top}\BPi(\BZ\otimes \I_p) = (\BZ\otimes \I_p)^{\top}\BPi(\BZ\otimes \I_p) = \sum_{i=1}^n\BPi_i = n \left( 1 - \frac{d+1}{2p}\right)\I_{pd}
\]
where $\BPi_i\in\RR^{pd\times pd}$ is the projection onto the tangent space at $\BS_i.$ 

For any $\be_k\widetilde{\be}_{\ell}^{\top}$ where $\{\be_k\}_{k=1}^d$ and $\{\widetilde{\be}_{\ell}\}_{\ell=1}^p$ are the canonical basis of $\RR^d$ and $\RR^p$ respectively, then
\begin{align*}
\sum_{i=1}^n \lag   \BPi_i(\widetilde{\be}_{\ell}\otimes \be_k), \widetilde{\be}_{\ell}\otimes \be_k \rag & = \sum_{i=1}^n \left\lag \be_k\widetilde{\be}_{\ell}^{\top} - \frac{1}{2}(\BS_i\widetilde{\be}_{\ell}\be_k^{\top} + \be_k\widetilde{\be}_{\ell}^{\top}\BS_i^{\top})\BS_i, \be_k\widetilde{\be}_{\ell}^{\top}\right\rag \\
& = n - \frac{1}{2}\sum_{i=1}^n\left( (\be_k^{\top}\BS_i\widetilde{\be}_{\ell})^2 + \|\be_k\widetilde{\be}_{\ell}^{\top}\BS_i\|_F^2 \right)  \\
& = n - \frac{n(d+1)}{2p}
\end{align*}
where
\[
\sum_{i=1}^n (\be_k^{\top}\BS_i\widetilde{\be}_{\ell})^2 = \frac{n}{p},~~~\sum_{i=1}^n \lag \be_k\widetilde{\be}_{\ell}^{\top}\BS_i^{\top}, \be_k\widetilde{\be}_{\ell}^{\top}\BS_i^{\top}\rag = \frac{nd}{p}.
\]
Also 
\begin{align*}
\sum_{i=1}^n \lag   \BPi_i(\widetilde{\be}_{\ell}\otimes \be_k), \widetilde{\be}_{\ell'}\otimes \be_{k'} \rag 
& = \sum_{i=1}^n \left\lag \be_k\widetilde{\be}_{\ell}^{\top} - \frac{1}{2}(\BS_i\widetilde{\be}_{\ell}\be_k^{\top} + \be_k\widetilde{\be}_{\ell}^{\top}\BS_i^{\top})\BS_i, \be_{k'}\widetilde{\be}_{\ell'}^{\top}\right\rag \\
& = - \frac{1}{2}\sum_{i=1}^n \lag \BS_i\widetilde{\be}_{\ell'},\be_k \rag \lag \BS_i\widetilde{\be}_{\ell}, \be_{k'}\rag = -\frac{1}{2}\sum_{i=1}^n (\BS_i)_{k\ell'}(\BS_i)_{k'\ell} =0
\end{align*}
where $\BS^{\top}\BS = \frac{nd}{p}\I_p$ and $\lag\be_k\widetilde{\be}_{\ell}^{\top}, \be_{k'}\widetilde{\be}_{\ell'}^{\top} \rag = 0.$
Here $(\BS_i)_{k\ell'}(\BS_i)_{k'\ell} \neq 0$ if and only if $\ell' - k = \ell - k'\mod p$. In this case, the sum is zero but the set of $\{\BS_i\}_{i=1}^n$ is symmetric, i.e., $\diag(\pm 1,\cdots,\pm1)\BS_i$ are in the set $\{\BS_i\}_{i=1}^n.$

Therefore, we have
\[
n^{-1}\sum_{i=1}^n \BPi_i = \left( 1-\frac{d+1}{2p}\right)\I_{pd}.
\]
In summary, we have
\begin{align*}
& n^{-1}\|\BPi (\BZ\BZ^{\top}\otimes \I_p)\BPi\| = 1 - \frac{d+1}{2p}, \\
& \|\BPi [\I_p\otimes \BS_1^{\top},\cdots,\I_p\otimes \BS_n^{\top}]^{\top}[\I_p\otimes \BS_1^{\top},\cdots,\I_p\otimes \BS_n^{\top}]\BPi\| \leq \frac{nd}{p}, \\
& [\I_p\otimes \BS_1^{\top},\cdots,\I_p\otimes \BS_n^{\top}]\BPi (\BZ\otimes \I_p)  = 0.
\end{align*}
Thus, applying the three inequalities to~\eqref{eq:opt_ss}, we conclude that $\BS$ is a second-order critical point if
\[
1 \geq (1 + t^{-1}) \left(1 -\frac{d+1}{2p}\right) \Longleftrightarrow t \geq  \frac{2p}{d+1} -1.
\]

\end{proof}


\end{document}